\documentclass[final,leqno,onefignum,onetabnum]{siamltex1213}
\usepackage[leqno]{amsmath}
\usepackage{amsfonts}
\usepackage{subfigure}
\usepackage{graphicx}
\usepackage{color}
\newtheorem{hypothesis}{Hypothesis}[section]

\newtheorem{remark}[theorem]{Remark}
\usepackage{showkeys}

\title{Symplectic Runge-Kutta Semi-discretization for Stochastic Schr\"odinger Equation\thanks{This work was
supported  by National Natural Science Foundation of China (NO. 91130003, NO. 11021101 and NO. 11290142)}}

\author{Chuchu Chen\thanks{State Key Laboratory of Scientific and Engineering Computing, Institute of Computational Mathematics and Scientific/Engineering Computing,
Academy of Mathematics and Systems
Science, Chinese Academy of Sciences, P.O. Box 2719, Beijing 100190, PR China. ({\tt chenchuchu@lsec.cc.ac.cn})({\tt hjl@lsec.cc.ac.cn}). }
\and Jialin Hong$^{\dag}$}

\begin{document}
\maketitle
\slugger{mms}{xxxx}{xx}{x}{x--x}

\begin{abstract}
Based on a variational principle with a stochastic forcing, we indicate that the stochastic Schr\"odinger equation in Stratonovich sense is an infinite-dimensional stochastic Hamiltonian system, whose phase flow preserves symplecticity.
We propose a general class of stochastic symplectic Runge-Kutta methods in temporal direction to the stochastic Schr\"odinger equation in Stratonovich sense and show that the methods preserve the charge conservation law. We present a convergence theorem on the relationship between the mean-square convergence order of a semi-discrete method and its local accuracy order. Taking stochastic midpoint scheme as an example of stochastic symplectic Runge-Kutta methods in temporal direction, based on the theorem we show that the mean-square convergence order of the semi-discrete scheme is 1 under appropriate assumptions.
\end{abstract}

\begin{keywords} stochastic Schr\"odinger equation, infinite-dimensional stochastic Hamiltonian system, symplectic structure, symplectic Runge-Kutta method,  semi-discretization, mean-square convergence order.\end{keywords}

\begin{AMS}{ 65C20, 65C30, 65C50.}\end{AMS}

\pagestyle{myheadings}
\thispagestyle{plain}
\markboth{C.~CHEN and J.~HONG}{SYMPLECTIC RK SEMI-DISCRETIZATION FOR SSE}

\section{Introduction}
Schr\"odinger equation, as one of the basic models for wave propagation, plays an essential role in various fields such as optics, quantum physics, optical fiber communications, plasma physics (see \cite{Sulem} and references therein).
Recently, stochastic perturbations of this equation have been investigated (\cite{Bang1,Bang2,Bass,Elgin}), where the perturbations may due to the inhomogeneous media or the thermal fluctuations, etc.
More recently, some attentions have been paid to the studies of stochastic Schr\"odinger equation
from both theoretical and numerical views;
 see \cite{Bouard2004,Bouard2006,Bouard2003,Jiang2012,Liu}.
The local and global existences of solution in space ${\mathbb H}^{1}({\mathbb R}^{n})$ for stochastic Schr\"odinger equation are investigated in \cite{Bouard2003}.
There are some conservation laws for deterministic Schr\"odinger equation, for example, charge conservation law and energy conservation law. The evolution of these invariant quantities in the case of stochastic Schr\"odinger equation also are considered in \cite{Bouard2003}. It is well known that the deterministic Schr\"odinger equation is an infinite-dimensional Hamiltonian system, that is to say, its phase flow preserves symplecticity, see \cite{Faou,Hong2006} for details. \cite{Jiang2012} establishes the theory about stochastic multi-symplectic conservation law for stochastic Hamiltonian partial differential equations, the stochastic Schr\"odinger equation as a concrete example is investigated for this property. As far as we know, there has been no work concerning the infinite-dimensional stochastic Hamiltonian system, its stochastic symplecticity and symplectic semi-discretization up to now. In this paper we present the general formula of infinite-dimensional stochastic Hamiltonian system via a variational principle with a stochastic forcing and investigate the stochastic symplecticity of stochastic Schr\"odinger equation in Stratonovich sense, which is shown to be an infinite-dimensional stochastic Hamiltonian system.

For the numerical approximations of stochastic Schr\"odinger equation, a semi-discrete scheme for stochastic nonlinear Schr\"odinger equation in Stratonovich sense is proposed in \cite{Bouard2004}. Authors also obtain the convergence of the discrete solution in various topologies. For stochastic nonlinear Schr\"odinger equation in It\^o sense with power nonlinearity, \cite{Bouard2006} analyzes the error of a semi-discrete scheme and proves that the numerical scheme has strong order $\frac{1}{2}$ in general and order 1 if the noise is additive. Using the integral representation idea, \cite{Liu} proposes splitting schemes to stochastic nonlinear Schr\"odinger equation in Stratonovich sense and proves the first order convergence in non-global Lipschitz case.
 It is important to design a numerical scheme which preserves the properties of the original problems as much as possible. For Hamiltonian system, symplectic methods are shown to be superior to non-symplectic ones especially in long time computation, owing to their preservation of the qualitative property, the symplecticity of the underlying continuous differential equation systems; see \cite{Milstein1995,Wang2009} and references therein. Since the stochastic Schr\"odinger equation in Stratonovich sense is an infinite-dimensional stochastic Hamiltonian system, we investigate numerical methods to preserve symplecticity.

 In the present paper, we present a general class of Runge-Kutta methods in temporal direction for stochastic Schr\"odinger equation firstly and then obtain the symplectic conditions for Runge-Kutta methods in temporal direction.
 Under the symplectic conditions of Runge-Kutta methods, we show that they preserve the charge conservation law.
Adopting the idea of the work \cite{Milstein1995},
which establishes the mean-square order of convergence of a method resting on properties of its one-step approximation only for stochastic ordinary differential equations, we propose a convergence theorem on the mean-square orders of a class of semi-discrete schemes for stochastic Schr\"odinger equation allowing sufficient spatial regularity.
As a special case of stochastic symplectic Runge-Kutta methods applied to temporal discretization,
the mean-square convergence order of midpoint scheme is analyzed. Based on the convergence theorem, it is shown that the mean-square convergence order of the stochastic midpoint scheme is $1$ under appropriate assumptions.

In section 2, we present some preliminaries on stochastic Schr\"odinger equation. We derive the general formula of infinite-dimensional stochastic Hamiltonian system via a variational principle with a stochastic forcing and show that the phase flow of stochastic Schr\"odinger equation preserves symplecticity. In section 3, we propose a general class of stochastic symplectic Runge-Kutta methods in temporal direction for stochastic Schr\"odinger equation and show that they also preserve the charge conservation law. The midpoint scheme as an example of stochastic symplectic Runge-Kutta methods in temporal direction is analyzed. In section 4 a convergence theorem on the mean-square orders of a class of semi-discrete schemes for stochastic Schr\"odinger equation is proposed. Furthermore we obtain the mean-square convergence order of the midpoint scheme as an application of the theorem. At last, in section 5, numerical experiments are provided to validate theoretical results.

\section{Stochastic Schr\"odinger equation}
We are interested in the one-dimensional stochastic Schr\"odinger equation with multiplicative noise in the sense of Stratonovich in the domain $ [-L,L]\times[t_{0},T]$,
\begin{align}\label{eq1}
  &\mathbf{i}d_{t}\psi(x,t)+\Big(\frac{\partial^{2}\psi}{\partial x^{2}}+\Psi_{|\psi|^{2}}^{\prime}(|\psi|^{2},x,t)\psi(x,t)\Big)dt=\varepsilon\psi(x,t)\circ dW(t),\\
  &\psi(-L,t)=\psi(L,t)=0,
  \qquad \psi(x,t_{0})=\varphi(x),\nonumber
\end{align}
 where $\mathbf{i}$ is the imaginary unit, the solution $\psi$ is a ${\mathbb C}$-valued random field, $\frac{\partial^2 \psi}{\partial x^2}$ means the second derivative of function $\psi$ with respect to variable $x$, $\Psi(|\psi|^{2},x,t)$ is a real function of $(\psi,x,t)$,
 $\Psi_{|\psi|^{2}}^{\prime}(|\psi|^{2},x,t)$ means the derivative with respect to $|\psi|^2$,
  $\varepsilon\in{\mathbb R}$ is a parameter describing the scale of noise,
 and $W(t)$ is an infinite-dimensional ${\mathbf Q}$-Wiener process which will be specified below. Moreover, we assume that the solution of \eqref{eq1} exists globally (see \cite{Bouard2003} for the studies of existence for the solution of stochastic Schr\"odinger equation). All the analysis in this paper could go through to $d$-dimensional problem.

We rewrite equation (\ref{eq1}) as a stochastic evolution equation in abstract form
\begin{align}\label{eq2}
 & \mathbf{i}d\psi(t)+\big(A\psi(t)+F(\psi(t))\big)dt=G(\psi(t))\circ dW(t),\qquad \psi(t_{0})=\varphi,
\end{align}
where $A=\frac{\partial^{2}}{\partial x^{2}}$ with Dirichlet boundary condition, $F(\psi(t))(x)=\Psi_{|\psi|^{2}}^{\prime}(|\psi|^{2},x,t)\psi(x,t)$ and $\big(G(\psi(t))(u)\big)(x)=\varepsilon\psi(x,t) u(x)$.
 From \cite {Jentzen2012} we know that the equation (\ref{eq2}) satisfies the condition of commutative noise in infinite dimension, which is
 $$
 \Big(G^{\prime}(\psi(t))\big(G(\psi(t))u\big)\Big)(\tilde{u})(x)= \Big(G^{\prime}(\psi(t))\big(G(\psi(t))\tilde{u}\big)\Big)(u)(x)=\varepsilon^2\psi(x,t) u(x)\tilde{u}(x).
 $$

$W(t)$ denotes a ${\mathbf Q}$-Wiener process on the Hilbert space $U={\mathbb L}^{2}([-L,L],\mathbb{R})$ which is the subspace of $U_c={\mathbb L}^{2}([-L,L],\mathbb{C})$ consisting of real valued square integrable functions. Here ${\mathbf Q}\in {\mathcal L}(U)$ is nonnegative, symmetric and with finite trace. Let $\{e_{k}\}_{k\in\mathbb{N}}$ be an orthonormal basis of $U$ consisting of eigenvectors of ${\mathbf Q}$.
Then there is a sequence of independent real-valued Brownian motions $\{\beta_{k}\}_{k\in\mathbb{N}}$ on the probability space $(\Omega,\mathcal{F},P)$ such that
$W(t)=\sum_{k\in\mathbb{N}}\beta_{k}(t){\mathbf Q}^{\frac12}e_{k},$
with covariance operator ${\mathbf Q}={\mathbf Q}^{\frac12}\circ {\mathbf Q}^{\frac12}$; we refer to \cite{Prato1992,Prevot2007} for more information about the infinite-dimensional Wiener process.

In addition, ${\mathbf Q}^{\frac12}$ is a Hilbert-Schmidt operator from $U$ to $U$ (the space of all Hilbert-Schmidt operators from $U$ to another Hilbert space $H$ is denoted by $HS(U,H)$). In fact,
$
  |{\mathbf Q}^{\frac12}|_{HS(U,U)}=\sum_{k\in\mathbf{N}}\langle{\mathbf Q}^{\frac12}e_{k},{\mathbf Q}^{\frac12}e_{k}\rangle_{U}
  =\sum_{k\in\mathbf{N}}\langle{\mathbf Q}e_{k},e_{k}\rangle_{U}
  =tr(\mathbf {Q}).
$
More assumptions on ${\bf Q}$ will be specified later.

Here we mention an important property of the infinite dimensional stochastic integral, which contributes a lot in the mean-square error estimation:
for any predictable stochastic process $G_s$ satisfying ${\bf E}\int_{t_0}^{T}|G_{s}{\mathbf Q}^{\frac12}|_{HS(U,H)}^{2}ds<\infty$, we have
\begin{equation}\label{eq5}
  {\bf E}\Big|\int_{t_0}^{T}G_{s}dW_{s}\Big|_{H}^{2}
  ={\bf E}\int_{t_0}^{T}|G_{s}{\mathbf Q}^{\frac12}|_{HS(U,H)}^{2}ds,
\end{equation}
where the stochastic integral on the left-hand side is of It\^o sense and $\mathbf{E}(\cdot)$ means the expectation.

As pointed out in \cite{Bouard2003}, there is an equivalent It\^o form of equation (\ref{eq2})
\begin{align}\label{eq3}
 & \mathbf{i}d\psi+(A\psi+\tilde{F}(\psi))dt=G(\psi) dW(t),\quad \psi(t_{0})=\varphi,
\end{align}
where $\tilde{F}(\psi)=F(\psi)+\frac{\mathbf{i}}{2}\varepsilon^{2}\psi \aleph_{{\mathbf Q}}$ with
 $ \aleph_{{\mathbf Q}}(x)=\sum_{k\in \mathbb{N}}({\mathbf Q}^{\frac12}e_{k}(x))^{2},  \text{ for } x\in [-L,L]$. Here
 the series for the function $\aleph_{{\mathbf Q}}(x)$ converges.

 The mild solution of equation (\ref{eq3}) reads
 \begin{equation}\label{mild solution of S equation}
   \psi(t)=S(t)\varphi+\mathbf{i}\int_{t_{0}}^{t}S(t-r)\tilde{F}(\psi(r))dr-\mathbf{i}\int_{t_{0}}^{t}S(t-r)G(\psi(r))dW(r),
 \end{equation}
 where $S(t)_{t\in\mathbb{R}}$ denotes the group of solution operator of the deterministic linear differential equation
 $ \mathbf{i}\frac{d\psi}{dt}=\Delta \psi.$
We may use both equations \eqref{eq1} and \eqref{eq3} in the following analysis depending on which type of equation could provide much convenience;
it is more convenient to use equation in Stratonovich sense in dealing with symplecticity while It\^o equation provides much convenience when we make the error estimates.

If the size $\varepsilon$ of the noise equals to $0$, i.e., the noise term is eliminated, we get the deterministic Schr\"odinger equation.
In this case, it possesses some global conservation laws, for example,
 charge conservation law
$
  \mathcal{E}_{1}(\psi)=\int_{-L}^{L}|\psi|^{2}dx=\mathcal{E}_{1}(\varphi),
$
and energy conservation law if $\Psi$ is independent of $t$,
$
  \mathcal{E}_{2}(\psi)=\int_{-L}^{L}(|\psi_{x}|^{2}-\Psi(|\psi|^2,x))dx=\mathcal{E}_{2}(\varphi).
$
We refer to \cite{Faou,Hong2006} for further understanding.
These quantities are important criteria of measuring whether a numerical simulation is good or not.

In the Stratonovich sense, the charge conservation law is still preserved by the solution of equation (\ref{eq1}). In general, there is no energy conservation law for stochastic Schr\"odinger equation. One only can obtain the
 relationship satisfied by the averaged energy \cite[Proposition 4.5]{Bouard2003}. They are stated as follows.
\begin{proposition}\cite{Bouard2003}
  The stochastic nonlinear Schr\"{o}dinger equation (\ref{eq1}) possesses the charge conservation law $P$-a.s.,
  $
    \mathcal{E}_{1}(\psi(t))=\int_{-L}^{L}|\psi(t)|^{2}dx=\mathcal{E}_{1}(\varphi),\,\forall \,t\in(t_0,\,T].
  $
\end{proposition}

\begin{proposition}\cite{Bouard2003}
  The averaged energy $\mathbf{E}(\mathcal{E}_{2}(\psi(t)))$ satisfies
\begin{equation*}
  \mathbf{E}(\mathcal{E}_{2}(\psi(t)))=\mathbf{E}(\mathcal{E}_{2}(\varphi))
  +\frac{\varepsilon^{2}}{2}\int_{t_0}^{t}\int_{-L}^{L}|\psi(x,s)|^{2}
  \sum_{k\in \mathbf{N}}|\frac{\partial}{\partial x}{\mathbf Q}^{\frac12}e_{k}(x)|^{2}dxds  \nonumber
\end{equation*}
if $\Psi$ is independent of $t$.
\end{proposition}

\subsection{Preserving symplectic structure}
One of the inherent canonical properties of Hamiltonian system is the symplecticity of its flow (cf. \cite{Arnold}). The analysis of symplecticity of stochastic
Hamiltonian system is presented in \cite{Bismut,Milstein1995} for instance. As we all know, the deterministic Schr\"odinger equation can be rewritten as an infinite-dimensional Hamiltonian system with a classical Hamiltonian structure \cite{Bridge}. In this part, we will focus on the corresponding symplectic analysis of stochastic Schr\"odinger equation. Firstly, we present the general formula of infinite-dimensional stochastic Hamiltonian system via a variational principle with a stochastic forcing; see \cite[Chapter 4.1]{Wang2007} for the application of the generalized variational principle to stochastic ordinary system.

Given two functionals ${\mathcal L}$ and ${\mathcal H_{2}}$, we define the generalized action integral by
\begin{equation*}
  \bar{\mathcal{S}}=\int_{t_{0}}^{T}\Big(\mathcal{L}(P,Q,\dot{P},\dot{Q})-\mathcal{H}_{2}(P,Q)\circ \dot{\chi}\Big) dt,
\end{equation*}
where functions $P, \,Q:[-L,\,L]\times[t_0,\,T]\rightarrow {\mathbb R}$, and $\dot{P}$, $\dot{Q}$ mean the derivatives with respect to time. By denoting $\dot{\chi}$ the time-space noise,
which is considered as the temporal derivative of the infinite-dimensional Wiener process $W(t)$, i.e. $\dot{\chi}=\frac{dW(t)}{dt}$,
we may consider
${\mathcal H_2}\circ\dot{\chi}$  formally as the work done by noise (non-conservative part).

The variation of the action integral follows as
\begin{align*}
  \delta \bar{\mathcal{S}}&=\delta \int_{t_{0}}^{T}\Big(\mathcal{L}(P,Q,\dot{P},\dot{Q})-\mathcal{H}_{2}(P,Q)\circ \dot{\chi}\Big) dt\\
  &=\int_{-L}^{L}\int_{t_0}^{T} \Big(\frac{\delta \mathcal{L}}{\delta P}\delta P+\frac{\delta \mathcal{L}}{\delta Q}\delta Q+\frac{\delta \mathcal{L}}{\delta \dot{P}}\delta \dot{P}+
  \frac{\delta \mathcal{L}}{\delta \dot{Q}}\delta \dot{Q}
 -\frac{\delta \mathcal{H}_{2}}{\delta P}\delta P\circ \dot{\chi}-\frac{\delta \mathcal{H}_{2}}{\delta Q}\delta Q\circ \dot{\chi}\Big)dtdx\nonumber\\
  &=\int_{-L}^{L}\int_{t_0}^{T} \Big[\Big(\frac{\delta \mathcal{L}}{\delta P}-\frac{d}{dt}\Big(\frac{\delta \mathcal{L}}{\delta \dot{P}}\Big)-\frac{\delta \mathcal{H}_{2}}{\delta P}\circ \dot{\chi}\Big)\delta P
+\Big(\frac{\delta \mathcal{L}}{\delta Q}-\frac{d}{dt}\Big(\frac{\delta \mathcal{L}}{\delta \dot{Q}}\Big)-\frac{\delta \mathcal{H}_{2}}{\delta Q}\circ \dot{\chi}\Big)\delta Q\Big]dtdx.\nonumber
\end{align*}
Here an integration by parts is performed and the boundary conditions $\delta P(x,T)=\delta P(x,t_0)=0$, $\delta Q(x,T)=\delta Q(x,t_0)=0$ are used.

Thus the Hamilton's principle
\[\delta \bar{\mathcal{S}}=\delta \int_{t_{0}}^{T}\Big(\mathcal{L}(P,Q,\dot{P},\dot{Q})-\mathcal{H}_{2}(P,Q)\circ \dot{\chi}\Big) dt=0\]
leads to
$
  \frac{d}{dt}\Big(\frac{\delta \mathcal{L}}{\delta \dot{P}}\Big)=\frac{\delta \mathcal{L}}{\delta P}-\frac{\delta \mathcal{H}_{2}}{\delta P}\circ \dot{\chi},\,
  \frac{d}{dt}\Big(\frac{\delta \mathcal{L}}{\delta \dot{Q}}\Big)=\frac{\delta \mathcal{L}}{\delta Q}-\frac{\delta \mathcal{H}_{2}}{\delta Q}\circ \dot{\chi}.
$
Introduce the Hamiltonian $\mathcal{H}_{1}(P,Q)$ and the generator $\mathcal{G}$ which are defined according to the equation
\begin{equation*}
  \mathcal{L}(P,Q,\dot{P},\dot{Q})-\int_{-L}^{L}P\dot{Q}dx+\mathcal{H}_{1}(P,Q)-\frac{d\mathcal{G}}{dt}=0.
\end{equation*}
As in \cite{Wang2007}, using this equation, we find
$
  \dot{P}=-\frac{\delta \mathcal{H}_{1}}{\delta Q}-\frac{\delta \mathcal{H}_{2}}{\delta Q}\circ \dot{\chi},\,
 \dot{Q}=\frac{\delta \mathcal{H}_{1}}{\delta P}+\frac{\delta \mathcal{H}_{2}}{\delta P}\circ \dot{\chi},
$
which can be rewritten in the form
\begin{align}\label{9}
  dP=-\frac{\delta \mathcal{H}_{1}}{\delta Q}dt-\frac{\delta \mathcal{H}_{2}}{\delta Q}\circ dW(t),\qquad
  dQ=\frac{\delta \mathcal{H}_{1}}{\delta P}dt+\frac{\delta \mathcal{H}_{2}}{\delta P}\circ dW(t).
\end{align}
\begin{remark}
If $\mathcal{H}_{1}=\mathcal{H}_{2}=\mathcal{H}$, then $\mathcal{H}$ is an invariant of system (\ref{9}). In fact,
\begin{align*}
d\mathcal{H}&=\int_{-L}^{L}\Big(\frac{\delta \mathcal{H}}{\delta P}dP+\frac{\delta \mathcal{H}}{\delta Q}dQ\Big)dx\\
&=\int_{-L}^{L}\Big(-\frac{\delta \mathcal{H}}{\delta P}\frac{\delta \mathcal{H}}{\delta Q}dt-\frac{\delta \mathcal{H}}{\delta P}\frac{\delta \mathcal{H}}{\delta Q}\circ dW+\frac{\delta \mathcal{H}}{\delta Q}\frac{\delta \mathcal{H}}{\delta P}dt+\frac{\delta \mathcal{H}}{\delta Q}\frac{\delta \mathcal{H}}{\delta P}\circ dW\Big)dx=0.
\end{align*}
\end{remark}

The stochastic Schr\"odinger equation (\ref{eq1}) can be written as a canonical infinite dimensional Hamiltonian system.
Denote the real  and imaginary  part of the solution of stochastic Schr\"odinger equation (\ref{eq1}) by $\Re(\psi(x,t))=P(x,t)$ and $\Im(\psi(x,t))=Q(x,t)$, respectively.
Then we have
\begin{align}\label{infite 2}
    dP&=-\frac{\delta \mathcal{H}_{1}}{\delta Q}dt-\frac{\delta \mathcal{H}_{2}}{\delta Q}\circ dW(t)
      =-(A Q+\Psi_{|\psi|^{2}}^{\prime}(P^{2}+Q^{2},x,t)Q)dt+\varepsilon Q\circ dW(t),\\
    dQ&=\frac{\delta \mathcal{H}_{1}}{\delta P}dt+\frac{\delta \mathcal{H}_{2}}{\delta P}\circ dW(t)
      =(A P+\Psi_{|\psi|^{2}}^{\prime}(P^{2}+Q^{2},x,t)P)dt-\varepsilon P\circ dW(t),\nonumber
\end{align}
with initial conditions
$P(0)=p=\Re(\varphi),\; Q(0)=q=\Im(\varphi),$
and Hamiltonians
$ \mathcal{H}_{1}(P,Q)=\frac12 \int_{-L}^{L}\Big((P_{x}^{2}+Q_{x}^{2})-\Psi(P^{2}+Q^{2},x,t)\Big)dx$
and
$ \mathcal{H}_{2}(P,Q)=-\frac{\varepsilon}{2}\int_{-L}^{L}(P^{2}+Q^{2})dx.$

The symplectic form for system (\ref{infite 2}) is given by
\begin{equation}\label{sy1}
\bar{\omega}(t)=\int_{-L}^{L}dP\wedge dQ \; dx,
\end{equation}
where the overbar on $\omega$ is a reminder that the two-form $dP\wedge dQ$ is integrated over the space. Preservation of the symplectic form \eqref{sy1} means that
the spatial integral of the oriented areas of projections onto the coordinate planes $(p,q)$ is an integral invariant \cite{Arnold}.
Note that the differentials in \eqref{infite 2} and \eqref{sy1} have different meanings. In \eqref{infite 2}, $P,\,Q$ are treated as functions of time, and $p,\,q$ are fixed functions, while in \eqref{sy1} the differential is made with respect to initial data $p,\,q$.

Using the formula of change of variables in differential forms, we obtain
\begin{align*}
    \bar{\omega}(t)&=\int_{-L}^{L}dP\wedge dQ \; dx
    =\int_{-L}^{L}\Big(\frac{\partial P}{\partial p}\frac{\partial Q}{\partial q}-\frac{\partial P}{\partial q}\frac{\partial Q}{\partial p}\Big)dp\wedge dq \; dx. \nonumber
\end{align*}
Hence, the phase flow of (\ref{infite 2}) preserves symplectic structure if and only if
\begin{equation*}
  \frac{d\bar{\omega}(t)}{dt}=\int_{-L}^{L}\frac{d}{dt}\Big(\frac{\partial P}{\partial p}\frac{\partial Q}{\partial q}-\frac{\partial P}{\partial q}\frac{\partial Q}{\partial p}\Big)dp\wedge dq \; dx=0.
\end{equation*}
Introduce notations
$P_{p}=\frac{\partial P}{\partial p},\, P_{q}=\frac{\partial P}{\partial q},\, Q_{p}=\frac{\partial Q}{\partial p},\, Q_{q}=\frac{\partial Q}{\partial q},$
and let $H_{1}(P,Q)=\frac12 \Psi(P^{2}+Q^{2},x,t)$.
We know that from the differentiability with respect to initial data of stochastic infinite-dimensional equations, $P_{p}$, $Q_{p}$, $P_{q}$ and $Q_{q}$ obey the following system (see \cite[Chapter 4]{FK} and \cite[Chapter 9]{Prato1992})
\begin{align}\label{partial}
  &dP_{p}=-\Big(A Q_{p}+\frac{\partial^{2} H_{1}}{\partial P \partial Q}P_{p}+\frac{\partial^{2} H_{1}}{\partial Q^{2}}Q_{p}\Big)dt+\varepsilon Q_{p}\circ dW(t), \quad P_{p}(t_{0})=I,\\
 & dQ_{p}=\Big(A P_{p}+\frac{\partial^{2} H_{1}}{\partial P^{2}}P_{p}+\frac{\partial^{2} H_{1}}{\partial P\partial Q}Q_{p}\Big)dt-\varepsilon P_{p}\circ dW(t),
  \quad Q_{p}(t_{0})=0,\nonumber\\
  &dP_{q}=-\Big(A Q_{q}+\frac{\partial^{2} H_{1}}{\partial P\partial Q}P_{q}+\frac{\partial^{2} H_{1}}{\partial Q^{2}}Q_{q}\Big)dt+\varepsilon Q_{q}\circ dW(t),
\quad  P_{q}(t_{0})=0,\nonumber\\
 & dQ_{q}=\Big(A P_{q}+\frac{\partial^{2} H_{1}}{\partial P^{2}}P_{q}+\frac{\partial^{2} H_{1}}{\partial P\partial Q}Q_{q}\Big)dt-\varepsilon P_{q}\circ dW(t),
\quad  Q_{q}(t_{0})=I.\nonumber
\end{align}
Due to (\ref{partial}), we get
\begin{align*}
    &d\Big(\frac{\partial P}{\partial p}\frac{\partial Q}{\partial q}-\frac{\partial P}{\partial q}\frac{\partial Q}{\partial p}\Big)
    =d\Big(P_{p}Q_{q}-P_{q}Q_{p}\Big)
    =d(P_{p})Q_{q}+d(Q_{q})P_{p}-d(P_{q})Q_{p}-d(Q_{p})P_{q}\\
    &\quad=\Big[-\Big(A Q_{p}+\frac{\partial^{2} H_{1}}{\partial P \partial Q}P_{p}+\frac{\partial^{2} H_{1}}{\partial Q^{2}}Q_{p}\Big)Q_{q}
    +\Big(A Q_{q}+\frac{\partial^{2} H_{1}}{\partial P\partial Q}P_{q}+\frac{\partial^{2} H_{1}}{\partial Q^{2}}Q_{q}\Big)Q_{p}\nonumber\\
   &\quad\quad +\Big(A P_{q}+\frac{\partial^{2} H_{1}}{\partial P^{2}}P_{q}+\frac{\partial^{2} H_{1}}{\partial P\partial Q}Q_{q}\Big)P_{p}
    -\Big(A P_{p}+\frac{\partial^{2} H_{1}}{\partial P^{2}}P_{p}+\frac{\partial^{2} H_{1}}{\partial P\partial Q}Q_{p}\Big)P_{q}\Big]dt\nonumber\\
  &\quad\quad+\varepsilon \Big(Q_{p}Q_{q}-Q_{q}Q_{p}-P_{p}P_{q}+P_{q}P_{p}\Big)\circ dW(t)\\
&\quad=\Big(-(A Q_{p})Q_{q}+(A Q_{q})Q_{p}+(A P_{q})P_{p}-(A P_{p})P_{q}\Big)dt.\nonumber
\end{align*}
Note that $A$ is Laplacian, then we have
\begin{align*}
    \frac{d}{dt}\bar{\omega}&=\int_{-L}^{L}\frac{d}{dt}\Big(\frac{\partial P}{\partial p}\frac{\partial Q}{\partial q}-\frac{\partial P}{\partial q}\frac{\partial Q}{\partial p}\Big)dp\wedge dq \; dx\\
    &=\int_{-L}^{L}\Big(-A Q_{p}Q_{q}+A Q_{q}Q_{p}+A P_{q}P_{p}-A P_{p}P_{q}\Big) dp\wedge dq \; dx\nonumber\\
   & =-\int_{-L}^{L}[d(A Q)\wedge dQ+d(A P)\wedge dP] \;dx
    =-\int_{-L}^{L}\frac{\partial}{\partial x}[d(Q_{x})\wedge dQ+d(P_{x})\wedge dP] \;dx.\nonumber
\end{align*}
From the zero boundary condition, we obtain that
$\frac{d}{dt}\bar{\omega}(t)=0.$
Thus we have the following theorem:
\begin{theorem}
  The phase flow of stochastic Schr\"odinger equation
  \[\mathbf{i}d\psi+(A \psi+F(\psi))dt=G(\psi)\circ dW,\]
  with $F(\psi)=\Psi_{|\psi|^{2}}^{\prime}(|\psi|^{2},x,t)\psi$ and $G(\psi)=\varepsilon\psi$,
  preserves the symplectic structure \[\bar{\omega}(t)=\int_{-L}^{L}dP\wedge dQ\;dx\]
  with $P$ (resp. $Q$) being the real (resp. imaginary) part of the solution $\psi$.
\end{theorem}
\section{Symplectic Runge-Kutta semi-discretization in temporal direction}
For stochastic ordinary differential equation, stochastic Runge-Kutta method is an important class of numerical methods, which has been widely investigated, see \cite{Burrage,Ma,Milstein1995,Rossler,Wang2009} and references therein for instance.

Here we introduce the uniform partition $0=t_{0}<t_{1}<\cdots<t_{N}=T$ for simplicity, let $\Delta t=t_{n+1}-t_{n}$, $n=0,1,\cdots, N-1$  and $\Delta W_{n}=W(t_{n+1})-W(t_{n})$.
Apply s-stage stochastic Runge-Kutta methods, which only depend on the increments of Wiener process, to the equation (\ref{infite 2}) in the temporal direction, where $P$ and $Q$ mean the real and imaginary parts of solution for stochastic Schr\"odinger equation, respectively, we get that
\begin{subequations}\label{RK method}
\begin{align}
    P_{i}=P^{n}-\Delta t\sum_{j=1}^{s}a_{ij}^{(0)}(AQ_{j}+\Psi^{\prime}_{j}Q_{j})+\Delta W_{n}\sum_{j=1}^{s}a_{ij}^{(1)}\varepsilon Q_{j},\label{RK method_1}\\
    Q_{i}=Q^{n}+\Delta t\sum_{j=1}^{s}a_{ij}^{(0)}(AP_{j}+\Psi^{\prime}_{j}P_{j})-\Delta W_{n}\sum_{j=1}^{s}a_{ij}^{(1)}\varepsilon P_{j},\label{RK method_2}\\
    P^{n+1}=P^{n}-\Delta t\sum_{i=1}^{s}b_{i}^{(0)}(AQ_{i}+\Psi^{\prime}_{i}Q_{i})+\Delta W_{n}\sum_{i=1}^{s}b_{i}^{(1)}\varepsilon Q_{i},\label{RK method_3}\\
    Q^{n+1}=Q^{n}+\Delta t\sum_{i=1}^{s}b_{i}^{(0)}(AP_{i}+\Psi^{\prime}_{i}P_{i})-\Delta W_{n}\sum_{i=1}^{s}b_{i}^{(1)}\varepsilon P_{i},\label{RK method_4}
   \end{align}
\end{subequations}
where $\Psi^{\prime}_{i}=\Psi_{|\psi|^{2}}^{\prime}(P_{i}^{2}+Q_{i}^{2},x,t_{n}+\sum_{j=1}^{s}a_{ij}^{(0)}\Delta t)$ for $i=1,\cdots,s$ and $n=0,\cdots,N-1$, $\big(a_{ij}^{(0)}\big)_{s\times s}$ and $\big(a_{ij}^{(1)}\big)_{s\times s}$
are $s\times s$ matrices of real elements while $\big(b_{1}^{(0)},\cdots,b_{s}^{(0)}\big)$ and $\big(b_{1}^{(1)},\cdots,b_{s}^{(1)}\big)$ are real vectors.

Under the assumption of existence of solution of \eqref{RK method}, we could obtain the following condition for \eqref{RK method} to preserve the discrete symplectic structure. We postpone the existence of solution to Proposition \ref{pro4}.
\begin{theorem}
  Methods (\ref{RK method}) preserve the discrete symplectic structure, i.e.,
  \[\bar{\omega}^{n+1}=\int_{-L}^{L}dP^{n+1}\wedge dQ^{n+1}\; dx=\int_{-L}^{L}dP^{n}\wedge dQ^{n}\; dx=\bar{\omega}^{n},\]
  if coefficients satisfy the following conditions: $\forall \; i,j=1,\cdots,s,$
  \begin{align}\label{condition}
      &b_{i}^{(0)}b_{j}^{(0)}=b_{i}^{(0)}a_{ij}^{(0)}+b_{j}^{(0)}a_{ji}^{(0)},
     &b_{i}^{(0)}b_{j}^{(1)}=b_{i}^{(0)}a_{ij}^{(1)}+b_{j}^{(1)}a_{ji}^{(0)},
     & b_{i}^{(1)}b_{j}^{(1)}=b_{i}^{(1)}a_{ij}^{(1)}+b_{j}^{(1)}a_{ji}^{(1)}.
  \end{align}
\end{theorem}
\begin{proof}
From \eqref{RK method_3} and \eqref{RK method_4}, we have
\begin{align}\label{num sym}
    &dP^{n+1}\wedge dQ^{n+1}-dP^{n}\wedge dQ^{n}\\
    &=-\Delta W_{n}^{2}\sum_{i,j=1}^{s}b_{i}^{(1)}b_{j}^{(1)}d\tilde{f}_{i}\wedge\tilde{g}_{j}-\Delta W_{n}\sum_{i=1}^{s}b_{i}^{(1)}\Big(dP^{n}\wedge d\tilde{g}_{i}-d\tilde{f}_{i}\wedge dQ^{n}\Big)\nonumber\\
     &\quad +\Delta t\sum_{i=1}^{s}b_{i}^{(0)}\Big(dP^{n}\wedge d(AP_{i}+g_{i})-d(AQ_{i}+f_{i})\wedge dQ^{n}\Big)\nonumber\\
    &\quad+\Delta t\Delta W_{n}\sum_{i,j=1}^{s}b_{i}^{(0)}b_{j}^{(1)}\Big(d(AQ_{i}+f_{i})\wedge d\tilde{g}_{j}+d\tilde{f}_{j}\wedge d(AP_{i}+g_{i})\Big)\nonumber\\
    &\quad -\Delta t^{2}\sum_{i,j=1}^{s}b_{i}^{(0)}b_{j}^{(0)}\Big(d(AQ_{i})\wedge d(AP_{j}+g_{j})+df_{i}\wedge d(AP_{j}+g_{j})\Big).\nonumber
\end{align}
where
$f_{j}=\Psi^{\prime}_{j}Q_{j}, \; \tilde{f}_{j}=\varepsilon Q_{j}, \; g_{j}=\Psi^{\prime}_{j}P_{j}, \; \tilde{g}_{j}=\varepsilon P_{j}.$
And from \eqref{RK method_1} and \eqref{RK method_2},
\begin{align}\label{med}
    &dP^{n}=dP_{i}+\Delta t\sum_{j=1}^{s}a_{ij}^{(0)}d(AQ_{j})+\Delta t\sum_{j=1}^{s}a_{ij}^{(0)}df_{j}-\Delta W_{n}\sum_{j=1}^{s}a_{ij}^{(1)}d\tilde{f}_{j},\\
    &dQ^{n}=dQ_{i}-\Delta t\sum_{j=1}^{s}a_{ij}^{(0)}d(AP_{j})-\Delta t\sum_{j=1}^{s}a_{ij}^{(0)}dg_{j}+\Delta W_{n}\sum_{j=1}^{s}a_{ij}^{(1)}d\tilde{g}_{j}.\nonumber
\end{align}
Substituting (\ref{med}) into (\ref{num sym}), we obtain
\begin{align*}
    &dP^{n+1}\wedge dQ^{n+1}-dP^{n}\wedge dQ^{n}
   =\Delta W_{n}\sum_{i=1}^{s}b_{i}^{(1)}\Big(d\tilde{f}_{i}\wedge dQ_{i}-dP_{i}\wedge d\tilde{g}_{i}\Big)\nonumber\\
   &\quad+\Delta t\sum_{i=1}^{s}b_{i}^{(0)}\Big(dP_{i}\wedge d(AP_{i})+dP_{i}\wedge dg_{i}-d(AQ_{i})\wedge dQ_{i}-df_{i}\wedge dQ_{i}\Big)\\
   &\quad +\Delta W_{n}^{2}\sum_{i,j=1}^{s}(b_{i}^{(1)}a_{ij}^{(1)}+b_{j}^{(1)}a_{ji}^{(1)}-b_{i}^{(1)}b_{j}^{(1)})d\tilde{f}_{i}\wedge\tilde{g}_{j}\\
   &\quad+\Delta t\Delta W_{n}\sum_{i,j=1}^{s}(b_{i}^{(0)}b_{j}^{(1)}-b_{i}^{(0)}a_{ij}^{(1)}-b_{j}^{(1)}a_{ji}^{(0)})\\
   &\quad\quad\quad\quad\Big(d(AQ_{i})\wedge d\tilde{g}_{j}+d\tilde{f}_{i}\wedge d\tilde{g}_{j}+d\tilde{f}_{j}\wedge d(AP_{i})+d\tilde{f}_{j}\wedge dg_{i}\Big)\\
    &\quad +\Delta t^{2}\sum_{i,j=1}^{s}(b_{i}^{(0)}a_{ij}^{(0)}+b_{j}^{(0)}a_{ji}^{(0)}-b_{i}^{(0)}b_{j}^{(0)})\\
   &\quad\quad\quad\quad \Big(d(AQ_{i}\wedge d(AP_{j})+d(AQ_{i})\wedge d g_{j}+df_{i}\wedge d(AP_{j})+df_{i}\wedge dg_{j})\Big).
\end{align*}
Recalling the conditions (\ref{condition}) and the expressions of functions $f_{i},\; \tilde{f}_{i},\; g_{i},\; \tilde{g}_{i}$, we have
\begin{align*}
  &  \int_{-L}^{L}dP^{n+1}\wedge dQ^{n+1} \; dx=\int_{-L}^{L}dP^{n}\wedge dQ^{n}\; dx\\
   &+\Delta t\sum_{i=1}^{s}b_{i}^{(0)}\int_{-L}^{L}\frac{\partial}{\partial x}(dP_{i}\wedge d(P_{i})_{x}-d(Q_{i})_{x}\wedge dQ_{i})\; dx
    =\int_{-L}^{L}dP^{n}\wedge dQ^{n}\; dx.\nonumber
\end{align*}
Thus the proof of the theorem is completed.
\end{proof}

The solution of stochastic symplectic Runge-Kutta methods (\ref{RK method}) exists in space $U_c={\mathbb L}^2([-L,L],{\mathbb C})$, and it preserves the discrete charge conservation law.
\begin{proposition}\label{pro4}
There exists a $U_c$-valued solution of
the symplectic methods (\ref{RK method})-(\ref{condition}). Moreover, it possesses the discrete charge conservation law, i.e.,
\[\int_{-L}^{L}|\phi^{n+1}|^{2}dx=\int_{-L}^{L}|\phi^{n}|^{2}dx,\quad \forall n=0,1,\cdots, N.\]
\end{proposition}
 \begin{proof}
 Firstly, we show the preservation of charge conservation law.
Recall that $\phi^{n}=P^{n}+\mathbf{i}Q^{n}$ and let $\phi_{i}=P_{i}+\mathbf{i}Q_{i}$, then methods (\ref{RK method}) can be
rewritten as
\begin{subequations}
\begin{align}
&\phi_{i}=\phi^{n}+\mathbf{i}\Delta t\sum_{j=1}^{s}a_{ij}^{(0)}h_{j}-\mathbf{i}\Delta W_{n}\sum_{j=1}^{s}a_{ij}^{(1)}\tilde{h}_{j},\label{rk phi i}\\
&\phi^{n+1}=\phi^{n}+\mathbf{i}\Delta t\sum_{i=1}^{s}b_{i}^{(0)}h_{i}-\mathbf{i}\Delta W_{n}\sum_{i=1}^{s}b_{i}^{(1)}\tilde{h}_{i},\label{rk phi n+1}
\end{align}
\end{subequations}
with $h_{i}=A\phi_{i}+\Psi_{i}^{\prime}\phi_{i}$ and $\tilde{h}_{i}=\varepsilon \phi_{i}$.

Multiplying equation (\ref{rk phi n+1}) with $\bar{\phi}^{n+1}+\bar{\phi}^{n}$ and taking real part, we get
\begin{align}\label{midl}
|\phi^{n+1}|^{2}=|\phi^{n}|^{2}
+\Re \Big\{\mathbf{i}\Delta t\sum_{i=1}^{s}b_{i}^{(0)}h_{i}(\bar{\phi}^{n+1}+\bar{\phi}^{n})
-\mathbf{i}\Delta W_{n}\sum_{i=1}^{s}b_{i}^{(1)}\tilde{h}_{i}(\bar{\phi}^{n+1}+\bar{\phi}^{n})\Big\}.
\end{align}
From equation (\ref{rk phi i}), we have the expression of $\phi^{n}$,
\begin{equation*}
\phi^{n}=\phi_{i}-\mathbf{i}\Delta t\sum_{j=1}^{s}a_{i j}^{(0)}h_{j}+\mathbf{i}\Delta W_{n}\sum_{j=1}^{s}a_{ij}^{(1)}\tilde{h}_{j}.
\end{equation*}
Combining this together with equation (\ref{rk phi n+1}) leads to
\begin{align*}
\bar{\phi}^{n+1}+\bar{\phi}^{n}&=2\bar{\phi}_{i}+\mathbf{i}2\Delta t\sum_{j=1}^{s}a_{i j}^{(0)}\bar{h}_{j}-\mathbf{i} 2\Delta W_{n}
\sum_{j=1}^{s}a_{i j}^{(1)}\bar{\tilde{h}}_{j}
-\mathbf{i}\Delta t\sum_{j=1}^{s}b_{j}^{(0)}\bar{h}_{j}+\mathbf{i}\Delta W_{n}\sum_{j=1}^{s}b_{j}^{(1)}\bar{\tilde{h}}_{j}.
\end{align*}
Therefore the second term of the right hand side of equation (\ref{midl}) becomes
\begin{align*}
&\Re \Big\{\mathbf{i}\Delta t\sum_{i=1}^{s}b_{i}^{(0)}h_{i}(\bar{\phi}^{n+1}+\bar{\phi}^{n})
-\mathbf{i}\Delta W_{n}\sum_{i=1}^{s}b_{i}^{(1)}\tilde{h}_{i}(\bar{\phi}^{n+1}+\bar{\phi}^{n})\Big\}\\
&=\Im \Big\{2\Delta t\sum_{i=1}^{s}b_{i}^{(0)}A\phi_{i}\bar{\phi}_{i}\Big\}
+\Delta t^{2}\sum_{i j=1}^{s}\Big(b_{i}^{(0)}b_{j}^{(0)}-b_{i}^{(0)}a_{\iota j}^{(0)}-b_{j}^{(0)}a_{ji}^{(0)}\Big)\Re(h_{i}\bar{h}_{j})\\
&\quad-2\Delta t\Delta W_{n}\sum_{ij=1}^{s}\Big(b_{i}^{(0)}b_{j}^{(1)}-b_{i}^{(0)}a_{\iota j}^{(1)}-b_{j}^{(1)}a_{ji}^{(0)}\Big)\Re(h_{i}\bar{\tilde{h}}_{j})\\
&\quad+\Delta W_{n}^{2}\sum_{ij=1}^{s}\Big(b_{i}^{(1)}b_{j}^{(1)}-b_{i}^{(1)}a_{ij}^{(1)}-b_{j}^{(1)}a_{ji}^{(1)}\Big)\Re(\tilde{h}_{i}\bar{\tilde{h}}_{j}),
\end{align*}
since $\Re \phi=\Re \bar{\phi}$.
Recalling the symplectic condition (\ref{condition}), we have
\begin{equation*}
  |\phi^{n+1}|^{2}=|\phi^{n}|^{2}+\Im \Big\{2\Delta t\sum_{i=1}^{s}b_{i}^{(0)}A\phi_{i}\bar{\phi}_{i}\Big\}.
\end{equation*}
Integrating the above equation from $-L$ to $L$ with respect to $x$ leads to the preservation of charge conservation law.

Secondly, the existence of solution follows from a standard Galerkin method and Brouwer's theorem; see \cite[Lemma 3.1]{Akrivis} and \cite[Lemma 3.1]{Bouard2004}.
From \eqref{rk phi n+1}, we have that
$$
\phi^{n+\frac12}=\phi^n+\frac{{\mathbf i}}{2}\Delta t\sum_{i=1}^{s}b_{i}^{(0)}h_{i}-\frac{\mathbf{i}}{2}\Delta W_{n}\sum_{i=1}^{s}b_{i}^{(1)}\tilde{h}_{i}.
$$
Thus define the mapping $\Pi:U_c\rightarrow U_c$ by
$$
\Pi(\phi^{n+\frac12})=\phi^{n+\frac12}-\phi^n-\frac{{\mathbf i}}{2}\Delta t\sum_{i=1}^{s}b_{i}^{(0)}h_{i}+\frac{\mathbf{i}}{2}\Delta W_{n}\sum_{i=1}^{s}b_{i}^{(1)}\tilde{h}_{i}.
$$
Similar proof leads to
$$
\Re\langle\Pi(\phi^{n+\frac12}),\phi^{n+\frac12}\rangle_{U_c}=\|\phi^{n+\frac12}\|_{U_c}^2-\Re\langle\phi^n,\phi^{n+\frac12}\rangle_{U_c}\geq
\|\phi^{n+\frac12}\|_{U_c}\big(\|\phi^{n+\frac12}\|_{U_c}-\|\phi^n\|_{U_c}\big).
$$
Hence for every $\phi^{n+\frac12}\in U_c$ such that $\|\phi^{n+\frac12}\|_{U_c}=\|\phi^n\|_{U_c}+1$, there holds $\Re\langle\Pi(\phi^{n+\frac12}),\phi^{n+\frac12}\rangle_{U_c}> 0$. The existence of solution follows from  a standard Galerkin method (make the above argument rigorous) and the Brouwer's theorem.
\end{proof}

\subsection{Midpoint scheme}
In the sequel, we consider a special case of symplectic Runge-Kutta methods. Let $s=1$ and take coefficients as
$a_{11}^{(0)}=a_{11}^{(1)}=\frac12 \text{ and } b_1^{(0)}=b_1^{(1)}=1.$
Obviously, coefficients satisfy symplectic conditions (\ref{condition}), and we obtain the midpoint scheme:
\begin{eqnarray}\label{midpoint scheme}
    P^{n+1}=P^{n}-\Delta t A Q^{n+\frac12}-\Delta t(\Psi^{\prime})^{n+\frac12}Q^{n+\frac12}+\varepsilon Q^{n+\frac12}\Delta W_{n},\\
    Q^{n+1}=Q^{n}+\Delta t A P^{n+\frac12}+\Delta t(\Psi^{\prime})^{n+\frac12}P^{n+\frac12}-\varepsilon P^{n+\frac12}\Delta W_{n}.\nonumber
\end{eqnarray}
Here $P^{n+\frac12}=\frac12(P^{n}+P^{n+1})$ and $Q^{n+\frac12}=\frac12(Q^{n}+Q^{n+1})$.
Denote the approximation of $\psi(t_{n})$ by $\phi^{n}$, and recall that $\phi^{n}=P^{n}+\mathbf{i}Q^{n}$ and $\phi^{n+\frac12}=\frac12(\phi^{n}+\phi^{n+1})$, we have
\begin{equation}\label{scheme2}
  \phi^{n+1}=\phi^{n}+\mathbf{i}\Delta tA \phi^{n+\frac12}+\mathbf{i}\Delta t F(t_{n+\frac12},\phi^{n+\frac12})-\mathbf{i}\varepsilon\phi^{n+\frac12}\Delta W_{n}.
\end{equation}
Of course, (\ref{scheme2}) is formal and has to be understood in the following sense
\begin{equation}\label{symp scheme}
  \phi^{n+1}=\hat{S}_{\Delta t}\phi^{n}+\mathbf{i}\Delta t T_{\Delta t}F(t_{n+\frac12},\phi^{n+\frac{1}{2}})-\mathbf{i}\varepsilon T_{\Delta t}\phi^{n+\frac{1}{2}}\Delta W_{n},
\end{equation}
where $F(t_{n+\frac12},\phi^{n+\frac{1}{2}}):=\Psi_{|\psi|^{2}}^{\prime}(|\phi^{n+\frac{1}{2}}|^{2},x,t_{n+\frac{1}{2}})\phi^{n+\frac{1}{2}},$ the operators are
\begin{equation}\label{hats}\hat{S}_{\Delta t}=(Id-\frac{\mathbf{i}\Delta t}{2}A)^{-1}(Id+\frac{\mathbf{i}\Delta t}{2}A) \text{ and }
T_{\Delta t}=(Id-\frac{\mathbf{i}\Delta t}{2}A)^{-1}.\end{equation}
Since midpoint scheme is a special case of symplectic Runge-Kutta methods, it also preserves the discrete charge conservation law.
\begin{corollary}\label{propo1}
For midpoint scheme (\ref{midpoint scheme}), the discrete charge conservation law is preserved, i.e., $P$-a.s.,
\begin{equation*}
  \int_{-L}^{L}|\phi^{n+1}|^{2}dx=\int_{-L}^{L}|\phi^{n}|^{2}dx,\quad \forall n=0,1,\cdots, N.
\end{equation*}
\end{corollary}
The following propositions are discrete versions of the evolutionary relationship of average energy for midpoint scheme. For the general Runge-Kutta methods, it is an open problem.
\begin{proposition}\label{propo2}
  If $\Psi_{|\psi|^{2}}^{\prime}(|\psi|^{2},x,t)= V(x)$, then
  \begin{equation*}
\mathbf{E}(  \mathcal{E}_{2}(\phi^{n+1}))+\frac{\varepsilon}{\Delta t}  \int_{-L}^{L}\mathbf{E}(|\phi^{n+1}|^{2}\Delta W_{n})dx
  =\mathbf{E}(\mathcal{E}_{2}(\phi^{n})),
  \end{equation*}
  where
  $  \mathcal{E}_{2}(\phi^{n})=\int_{-L}^{L}|\nabla \phi^{n}|^{2}dx-\int_{-L}^{L}V(x)|\phi^{n}|^{2}dx.$
\end{proposition}
\begin{proof}
  Since scheme
 (\ref{symp scheme}) can be rewritten as
  \begin{equation*}
    \frac{\mathbf{i}}{\Delta t}(\phi^{n+1}-\phi^{n})+A \phi^{n+\frac{1}{2}}+\Psi_{|\psi|^{2}}^{\prime}(|\phi^{n+\frac{1}{2}}|^{2},x,t^{n+\frac{1}{2}})\phi^{n+\frac{1}{2}}
    =\frac{\varepsilon}{\Delta t}\phi^{n+\frac{1}{2}}\Delta W_{n},\nonumber
  \end{equation*}
  multiplying it with $\bar{\phi}^{n+1}-\bar{\phi}^{n}$, taking the real part and integrating with respect to $x$ in the space $[-L,L]$,
  the first term in the left hand side of the above equation becomes
  \begin{equation*}
    \Im\int_{-L}^{L}\frac{1}{\Delta t}(\phi^{n+1}-\phi^{n})(\bar{\phi}^{n+1}-\bar{\phi}^{n})dx
    =-\frac{1}{\Delta t}\Im\int_{-L}^{L}(\phi^{n+1}\bar{\phi}^{n}+\phi^{n}\bar{\phi}^{n+1})dx=0,\nonumber
  \end{equation*}
  the second term has the form
  \begin{align*}
      \Re\int_{-L}^{L}A \phi^{n+\frac{1}{2}}(\bar{\phi}^{n+1}-\bar{\phi}^{n})dx
     & =\frac{1}{2}\Re\int_{-L}^{L}\Big(-|\nabla \phi^{n+1}|^{2}+|\nabla\phi^{n}|^{2}+\nabla \phi^{n+1}\nabla\bar{\phi}^{n}-\nabla\phi^{n}\nabla\bar{\phi}^{n+1}\Big)dx\\
     &=-\frac{1}{2}\int_{-L}^{L}|\nabla \phi^{n+1}|^{2}dx
     +\frac{1}{2}\int_{-L}^{L}|\nabla \phi^{n}|^{2}dx,\nonumber
  \end{align*}
  the third term acquires the form
  \begin{align*}
      &\frac12 \Re\int_{-L}^{L}V(x)(\phi^{n+1}+\phi^{n})(\bar{\phi}^{n+1}-\bar{\phi}^{n})dx
      =\frac12 \int_{-L}^{L}V(x)|\phi^{n+1}|^{2}dx-\frac12 \int_{-L}^{L}V(x)|\phi^{n}|^{2}dx,\nonumber
  \end{align*}
  at last, the term turns into
  \begin{align*}
      &\Re\int_{-L}^{L}\Big(\frac{\varepsilon}{\Delta t}\Delta W_{n}\phi^{n+\frac12}(\bar{\phi}^{n+1}-\bar{\phi}^{n})\Big)dx
      =\frac{\varepsilon}{2\Delta t}\int_{-L}^{L}|\phi^{n+1}|^{2}\Delta W_{n}dx
      -\frac{\varepsilon}{2\Delta t}\int_{-L}^{L}|\phi^{n}|^{2}\Delta W_{n}dx.\nonumber
  \end{align*}
Combine the above equations, the result of the proposition is obtained.
\end{proof}
\begin{proposition}\label{propo3}
  For the general $\Psi_{|\psi|^{2}}^{\prime}(|\psi|^{2},x,t)$, we have the implicit relationship
  \begin{align*}
    \int_{-L}^{L}|\nabla \phi^{n+1}|dx-
    \int_{-L}^{L}|\nabla \phi^{n}|dx
    -\int_{-L}^{L}(\Psi^{\prime})^{n+\frac{1}{2}}(|\phi^{n+1}|^{2}-|\phi^{n}|^{2})dx\\
    =-\varepsilon \int_{-L}^{L}(|\phi^{n+1}|^{2}-|\phi^{n}|^{2})\Delta W_{n}dx.\nonumber
  \end{align*}
\end{proposition}
\begin{proof}
  The proof of this proposition is the same as that of Proposition \ref{propo2}.
\end{proof}
\section{Convergence theorem}
To investigate the mean-square convergence order of semi-discrete approximations to stochastic Schr\"odinger equation, we establish a convergence theorem on the relationship between local accuracy order and global mean-square order of a semi-discrete method.

In this subsection, we consider stochastic Schr\"odinger equation in It\^o sense
\begin{equation}\label{equation Ito form}
  \mathbf{i}d\psi+(A\psi+F(\psi))dt=G(\psi)dW,
\end{equation}
whose mild solution is
\begin{align}\label{mild solution}
  \psi(t_{n+1})&=S(\Delta t)\psi(t_{n})+\mathbf{i}\int_{t_{n}}^{t_{n+1}}S(t_{n+1}-r)F(\psi(r))dr
-\mathbf{i}\int_{t_{n}}^{t_{n+1}}S(t_{n+1}-r)G(\psi(r))dW(r)\nonumber\\
  &:=S(\Delta t)\psi(t_{n})+\Upsilon(\psi(r)).\nonumber
\end{align}
Denote the approximation of the solution $(\psi(t_{n}),\mathcal{F}_{t_{n}})$ for equation (\ref{equation Ito form}) by $(\phi^{n},\mathcal{F}_{t_{n}})$, which means that the numerical approximation $\phi^{n}$ is also $\mathcal{F}_{t_{n}}$-measurable. Here $\mathcal{F}_{t}$ is the filtration generated by Wiener process and initial value. Define adapted numerical solution $\phi^{n}$ recurrently by
\begin{equation}\label{numer}
  \phi^{n+1}=\hat{S}_{\Delta t}\phi^{n}+\Gamma(\phi^{n},\phi^{n+1},\Delta t,\Delta W_{n},T_{\Delta t}),
\end{equation}
where $\hat{S}_{\Delta t}$ and $T_{\Delta t}$ are two operators depending on operator $A$, such that $|\hat{S}_{\Delta t}|_{{\mathcal L}(U_c,U_c)}\leq 1$, which don't have to be the same as \eqref{hats} as long as they satisfy the mentioned hypotheses.

Denote the usual Sobolev space ${\mathbb H}^{s}:={\mathbb H}^{s}([-L,L];\mathbb{C})$ for $s\in \mathbb{R}$. We write $|\cdot|_{s}=|\cdot|_{{\mathbb H}^{s}}$ and
 make the following Hypotheses for equation (\ref{equation Ito form}).
\begin{hypothesis}\label{hypo1}
  For $\alpha\geq 0$, assume that $F:{\mathbb H}^{\alpha}\rightarrow {\mathbb H}^{\alpha}$ and $G(\cdot){\bf Q}^{\frac12}:{\mathbb H}^{\alpha}\rightarrow HS(U;{\mathbb H}^{\alpha})$ are globally Lipschitz functions.
\end{hypothesis}

Since the group $S(t)$ has no smooth effects on the solution, we assume that the solution has extra regularity. Let $\beta>\alpha$.
\begin{hypothesis}\label{hypo2}
  The solution $\psi$ is in $L^{2}(\Omega; L^{2}(0,T;{\mathbb H}^{\beta}))$, $F(\psi)$ is in $L^{2}(\Omega; L^{2}(0,T;{\mathbb H}^{\beta}))$, $G(\psi){\bf Q}^{\frac12}$ is in $L^{2}(\Omega; L^{2}(0,T;HS(U;{\mathbb H}^{\beta})))$
   and initial data $\varphi \in L^{2}(\Omega; {\mathbb H}^{\beta})$.
\end{hypothesis}

The difference of two operators $S(\Delta t)$ and $\hat{S}_{\Delta t}$ is essential in the mean-square convergence estimation. Furthermore, we assume that the order of approximation to the group $S(t)$ is of order $q$, which is stated below.
Throughout this paper, all constants $K$ depend on the coefficients of equation (\ref{equation Ito form}), the operator ${\mathbf Q}^{\frac12}$ and initial value $\varphi$ and do not depend on $n$ and $\Delta t$, which may be different from line to line.
\begin{hypothesis}\label{hypo3}
There exist constants $K>0$ and $q>0$ such that
  \begin{equation}
    |S(t_{k})-\hat{S}_{\Delta t}^{k}|_{\mathcal{L}({\mathbb H}^{\beta},{\mathbb H}^{\alpha})}\leq K\Delta t^{q}.\nonumber
  \end{equation}
\end{hypothesis}

\vspace{-6mm}
\cite{Milstein1995} establishes the mean-square order of convergence of a method resting on properties of its one-step approximation only for stochastic ordinary differential equations.
  The most successful point in the proof of the fundamental convergence theorem is the usage of conditional expectation.
  Properties of stochastic partial differential equations are different from those of stochastic ordinary differential equations, for example,
  in the mild solution forms of the continuous and discrete solutions, there are two different deterministic operators $S(\Delta t)$ and $\hat{S}_{\Delta t}$.
 So we have to separate operator term and integral term apart.
  Here we solve it by introducing a temporary process $\tilde{\psi}(r)$, $t_{n}\leq r \leq t_{n+1}$, which is a continuous approximation of $\psi(r)$:
\begin{align}\label{tilde psi}
   & \tilde{\psi}(r)=\hat{S}_{r-t_{n}}\tilde{\psi}(t_{n})+\mathbf{i}\int_{t_{n}}^{r}S(r-\rho)F(\tilde{\psi}(\rho))d\rho-\mathbf{i}\int_{t_{n}}^{r}S(r-\rho)G(\tilde{\psi}(\rho))dW(\rho)\\
 & \tilde{\psi}(x,0)=\varphi(x).\nonumber
\end{align}
If $r=t_{n+1}$, $\tilde{\psi}(t_{n+1})=\hat{S}_{\Delta t}\tilde{\psi}(t_{n})+\Upsilon(\tilde{\psi}(r))$.

Let $\phi$ be an $\mathcal{F}_{t_{n}}$-measurable random variable. $\tilde{\psi}_{t_{n},\phi}(t)$
denotes the solution of the system (\ref{tilde psi}) for $t_{n}\leq t\leq T$ satisfying the initial condition at $t=t_{n}$: $\tilde{\psi}(t_{n})=\phi$.
\begin{proposition}\label{lemma3}
  There exists a function $Z$ such that we have the representation
  \begin{equation}
    \tilde{\psi}_{t_{n},\tilde{\psi}(t_{n})}(t_{n+1})-\tilde{\psi}_{t_{n},\phi^{n}}(t_{n+1})=\hat{S}_{\Delta t}(\tilde{\psi}(t_{n})-\phi^{n})+Z.\nonumber
  \end{equation}
 And under Hypothesis \ref{hypo1}, we have
  \begin{align*}
   & \mathbf{E}|\tilde{\psi}_{t_{n},\tilde{\psi}(t_{n})}(t_{n+1})-\tilde{\psi}_{t_{n},\phi^{n}}(t_{n+1})|_{\alpha}^{2}\leq \mathbf{E}|\tilde{\psi}(t_{n})-\phi^{n}|_{\alpha}^{2}(1+K\Delta t),\nonumber\\
   & \mathbf{E}|Z|_{\alpha}^{2}\leq K\Delta t \mathbf{E}|\tilde{\psi}(t_{n})-\phi^{n}|_{\alpha}^{2}.\nonumber
  \end{align*}
\end{proposition}
\begin{proof}
  Since for $\phi=\tilde{\psi}(t_n)$ or $\phi^n$:
 $   \tilde{\psi}_{t_{n},\phi}(t_{n+1})=\hat{S}_{\Delta t}\phi+\Upsilon(\tilde{\psi}_{t_{n},\phi}(r)),$
  we have
    \begin{equation}
    \tilde{\psi}_{t_{n},\tilde{\psi}(t_{n})}(t_{n+1})-\tilde{\psi}_{t_{n},\phi^{n}}(t_{n+1})=\hat{S}_{\Delta t}(\tilde{\psi}(t_{n})-\phi^{n})+Z,\nonumber
  \end{equation}
  where $$
  Z=\Upsilon(\tilde{\psi}_{t_{n},\tilde{\psi}(t_{n})}(r))-\Upsilon(\tilde{\psi}_{t_{n},\phi^{n}}(r))
  $$
  Then
  \begin{align}\label{eq4}
      &\mathbf{E}|\tilde{\psi}_{t_{n},\tilde{\psi}(t_{n})}(t_{n+1})-\tilde{\psi}_{t_{n},\phi^{n}}(t_{n+1})|_{\alpha}^{2}\\
     & =\mathbf{E}|\hat{S}_{\Delta t}(\tilde{\psi}(t_{n})-\phi^{n})|_{\alpha}^{2}+\mathbf{E}|Z|_{\alpha}^{2}+2\mathbf{E}\langle\hat{S}_{\Delta t}(\tilde{\psi}(t_{n})-\phi^{n}),Z\rangle_{\alpha}\nonumber\\
     &=\mathbf{E}|\hat{S}_{\Delta t}(\tilde{\psi}(t_{n})-\phi^{n})|_{\alpha}^{2}+\mathbf{E}|Z|_{\alpha}^{2}+2\mathbf{E}\langle\hat{S}_{\Delta t}(\tilde{\psi}(t_{n})-\phi^{n}),\mathbf{E}^{t_{n}}(Z)\rangle_{\alpha}\nonumber
  \end{align}
  where $\mathbf{E}^{t_{n}}(\cdot)$ denotes the conditional expectation with respect to $\mathcal{F}_{t_{n}}$.
  By Hypothesis \ref{hypo1} and $|S(t)|_{\mathcal {L}(U_c,U_c)}\leq 1$, we have
  \begin{align*}
    {\mathbf E}|Z|_{\alpha}^2\leq& 2{\mathbf E}\Big|\int_{t_{n}}^{t_{n+1}}S(t_{n+1}-r)\Big(F(\tilde{\psi}_{t_{n},\tilde{\psi}(t_n)}(r))-F(\tilde{\psi}_{t_n,\phi^n}(r))\Big)dr\Big|_{\alpha}^2\\
    &+2{\mathbf E}\Big|\int_{t_{n}}^{t_{n+1}}S(t_{n+1}-r)\Big(G(\tilde{\psi}_{t_{n},\tilde{\psi}(t_n)}(r))-G(\tilde{\psi}_{t_n,\phi^n}(r))\Big)dW(r)\Big|_{\alpha}^{2}\\
    \leq & K\int_{t_n}^{t_{n+1}}{\mathbf E}|\tilde{\psi}_{t_{n},\tilde{\psi}(t_n)}(r)-\tilde{\psi}_{t_n,\phi^n}(r)|_{\alpha}^2dr
  \end{align*}
For the third term on the right-hand side of \eqref{eq4}, by Young's inequality we have
\begin{align*}
  {\mathbf E}\langle\hat{S}_{\Delta t}(\tilde{\psi}(t_{n})-\phi^{n}),\mathbf{E}^{t_{n}}(Z)\rangle_{\alpha}
  \leq K\Delta t{\mathbf E}|\tilde\psi(t_n)-\phi^n|_{\alpha}^{2}+K\Delta t^{-1}{\mathbf E}|{\mathbf E}^{t_n}(Z)|_{\alpha}^{2}\\
  \leq K\Delta t{\mathbf E}|\tilde\psi(t_n)-\phi^n|_{\alpha}^{2}+ K\int_{t_n}^{t_{n+1}}{\mathbf E}|\tilde{\psi}_{t_{n},\tilde{\psi}(t_n)}(r)-\tilde{\psi}_{t_n,\phi^n}(r)|_{\alpha}^2dr
\end{align*}
Combining these into the right-hand side of \eqref{eq4} and via
 Gronwall's Lemma, the proposition is proved.
\end{proof}

We now propose the mean-square convergence theorem for numerical method (\ref{numer}).
\begin{theorem}\label{convergence theorem}
Assume that Hypotheses \ref{hypo1}, \ref{hypo2} and \ref{hypo3} hold.
 And there exists a positive constant $p$ such that the method (\ref{numer}) satisfies the following conditions:
    \begin{align}
  &\Big(\mathbf{E}|\mathbf{E}^{t_{n}}(\Upsilon(\tilde{\psi}_{t_{n},\phi^{n}}(r))-\Gamma(\phi^{n},\phi^{n+1},\Delta t,\Delta W_{n},T_{\Delta t}))|_{\alpha}^{2}\Big)^{\frac12}\leq K\Delta t^{p+1},\label{cond1}\\
  &\Big(\mathbf{E}|\Upsilon(\tilde{\psi}_{t_{n},\phi^{n}}(r))-\Gamma(\phi^{n},\phi^{n+1},\Delta t,\Delta W_{n},T_{\Delta t})|_{\alpha}^{2}\Big)^{\frac12}\leq K\Delta t^{p+\frac12}.\label{cond2}
  \end{align}
  Then for $n=1,2,\cdots, N$, we have
  \begin{equation}
    \Big(\mathbf{E}|\psi(t_{n})-\phi^{n}|_{\alpha}^{2}\Big)^{1/2} \leq K\Delta t^{\min \{q,p\}}.\nonumber
  \end{equation}
\end{theorem}
  Before the proof of theorem \ref{convergence theorem}, we firstly present a lemma which deals with the mean-square bound of numerical solution under the conditions (\ref{cond1})-(\ref{cond2}).
\begin{lemma}\label{lemma:num}
 Under the Hypothesis \ref{hypo1}, for all natural number $N$ and all $n=0,1,\cdots,N$, the following inequality holds:
  \begin{equation}
    \mathbf{E}|\phi^{n}|_{\alpha}^{2}\leq K(1+\mathbf{E}|\varphi|_{\alpha}^{2})\leq K.\nonumber
  \end{equation}
\end{lemma}
\begin{proof}
  From the definition of $\tilde{\psi}_{t_{n},\phi^{n}}(t_{n+1})$: $\tilde{\psi}_{t_{n},\phi^{n}}(t_{n+1})=\hat{S}_{\Delta t}\phi^{n}+\Upsilon(\tilde{\psi}_{t_{n},\phi^{n}}(r)),$
  one obtains easily
  \begin{align*}
      \mathbf{E}|\tilde{\psi}_{t_{n},\phi^{n}}(t_{n+1})|_{\alpha}^{2}&\leq \mathbf{E}|\phi^{n}|_{\alpha}^{2}+2\mathbf{E}\Big|\int_{t_{n}}^{t_{n+1}}S(t_{n+1}-r)F(\tilde{\psi}_{t_{n},\phi^{n}}(r))dr\Big|_{\alpha}^{2}\\
     &\quad+2\mathbf{E}\Big|\int_{t_{n}}^{t_{n+1}}S(t_{n+1}-r)G(\tilde{\psi}_{t_{n},\phi^{n}}(r))dW(r)\Big|_{\alpha}^{2}\\
      &\quad+2\mathbf{E}\langle\hat{S}_{\Delta t}\phi^{n},\int_{t_{n}}^{t_{n+1}}S(t_{n+1}-r)F(\tilde{\psi}_{t_{n},\phi^{n}}(r))dr\rangle_{\alpha}\\
      &\quad+2\mathbf{E}\langle\hat{S}_{\Delta t}\phi^{n},\int_{t_{n}}^{t_{n+1}}S(t_{n+1}-r)G(\tilde{\psi}_{t_{n},\phi^{n}}(r))dW(r)\rangle_{\alpha}.\nonumber
  \end{align*}
  Since
  \begin{align*}
    &\mathbf{E}\langle\hat{S}_{\Delta t}\phi^{n},\int_{t_{n}}^{t_{n+1}}S(t_{n+1}-r)G(\tilde{\psi}_{t_{n},\phi^{n}}(r))dW(r)\rangle_{\alpha}\\
   & \quad=\mathbf{E}\langle\hat{S}_{\Delta t}\phi^{n},\mathbf{E}^{t_{n}}\int_{t_{n}}^{t_{n+1}}S(t_{n+1}-r)G(\tilde{\psi}_{t_{n},\phi^{n}}(r))dW(r)\rangle_{\alpha}
 =0
  \end{align*}
  and functions $F$ and $G(\cdot){\bf Q}^{\frac12}$ are global Lipchitz which lead to linear growth property,
  we have
  \begin{equation*}
      \mathbf{E}|\tilde{\psi}_{t_{n},\phi^{n}}(t_{n+1})|_{\alpha}^{2}\leq (1+K\Delta t)\mathbf{E}|\phi^{n}|_{\alpha}^{2}+K\mathbf{E}\int_{t_{n}}^{t_{n+1}}|\tilde{\psi}_{t_{n},\phi^{n}}(r)|_{\alpha}^{2}dr.
  \end{equation*}
  Then Gronwall's inequality leads to
  $
  \mathbf{E}|\tilde{\psi}_{t_{n},\phi^{n}}(t_{n+1})|_{\alpha}^{2}\leq (1+K\Delta t)\mathbf{E}|\phi^{n}|_{\alpha}^{2}.
  $
  Similarly, one has the estimations
  \begin{align*}
      &\mathbf{E}|\tilde{\psi}_{t_{n},\phi^{n}}(t_{n+1})-\hat{S}_{\Delta t}\phi^{n}|_{\alpha}^{2}\leq K\Delta t(1+\mathbf{E}|\phi^{n}|_{\alpha}^{2}),\\
      &\mathbf{E}|\mathbf{E}^{t_{n}}(\tilde{\psi}_{t_{n},\phi^{n}}(t_{n+1})-\hat{S}_{\Delta t}\phi^{n})|_{\alpha}^{2}\leq K\Delta t^{2}(1+\mathbf{E}|\phi^{n}|_{\alpha}^{2}).
  \end{align*}
  Suppose that $\mathbf{E}|\phi^{n}|_{\alpha}^{2}<\infty$, then
  \begin{align*}
    \mathbf{E}|\phi^{n+1}|_{\alpha}^{2}\leq K\mathbf{E}|\tilde{\psi}_{t_{n},\phi^{n}}(t_{n+1})|_{\alpha}^{2}+K\mathbf{E}|\phi^{n+1}-\tilde{\psi}_{t_{n},\phi^{n}}(t_{n+1})|_{\alpha}^{2}
    <\infty.
  \end{align*}
  Since $\mathbf{E}|\varphi|_{\alpha}^{2}<\infty$, we have proved the existence of $\mathbf{E}|\phi^{n}|_{\alpha}^{2}, \; \forall\, n\in \{0,1,\cdots,N\}$.
  Thus from
  \[\phi^{n+1}=\Big(\phi^{n+1}-\tilde{\psi}_{t_n,\phi^n}(t_{n+1})\Big)+\Big(\tilde{\psi}_{t_n,\phi^n}(t_{n+1})-\hat{S}_{\Delta t}\phi^n\Big)+\hat{S}_{\Delta t}\phi^n,\]
  we have
  \begin{align*}
      &\mathbf{E}|\phi^{n+1}|_{\alpha}^{2} \leq \mathbf{E}|\phi^{n}|_{\alpha}^{2}+2\mathbf{E}|\Upsilon(\tilde{\psi}_{t_{n},\phi^n}(r))-\Gamma(\phi^{n},\phi^{n+1},\Delta t,\Delta W_{n},T_{\Delta t})|_{\alpha}^{2}\\
    &  \quad+\mathbf{E}\langle\hat{S}_{\Delta t}\phi^{n},\mathbf{E}^{t_{n}}(\Upsilon(\tilde{\psi}_{t_{n},\phi^n}(r))-\Gamma(\phi^{n},\phi^{n+1},\Delta t,\Delta W_{n},T_{\Delta t}))\rangle_{\alpha}\\
   &  \quad +\mathbf{E}\langle\hat{S}_{\Delta t}\phi^{n},\mathbf{E}^{t_{n}}(\tilde{\psi}_{t_{n},\phi^{n}}(t_{n+1})-\hat{S}_{\Delta t}\phi^{n})\rangle_{\alpha}+2\mathbf{E}|\tilde{\psi}_{t_{n},\phi^{n}}(t_{n+1})-\hat{S}_{\Delta t}\phi^{n}|_{\alpha}^{2}\\
    & \leq(1+K\Delta t)\mathbf{E}|\phi^{n}|_{\alpha}^{2}+K\Delta t.\nonumber
  \end{align*}
  Hence the proof is finished by induction.
\end{proof}

Now we are in the position of the proof of Theorem \ref{convergence theorem}.
\begin{proof}
  Firstly we consider the mean-square error between $\tilde{\psi}(t_{n+1})$ and $\phi^{n+1}$. Since
  \begin{align*}
     &\tilde{\psi}(t_{n+1})-\phi^{n+1}=\tilde{\psi}_{t_{n},\tilde{\psi}(t_{n})}(t_{n+1})-\phi^{n+1}\\
      &=\Big(\tilde{\psi}_{t_{n},\tilde{\psi}(t_{n})}(t_{n+1})-\tilde{\psi}_{t_{n},\phi^{n}}(t_{n+1})\Big)+
      \Big(\tilde{\psi}_{t_{n},\phi^{n}}(t_{n+1})-\phi^{n+1}\Big)\\
      &=\Big(\tilde{\psi}_{t_{n},\tilde{\psi}(t_{n})}(t_{n+1})-\tilde{\psi}_{t_{n},\phi^{n}}(t_{n+1})\Big)+\Big(\Upsilon(\tilde{\psi}_{t_{n},\phi^{n}}(r))-\Gamma(\phi^{n},\phi^{n+1},\Delta t,\Delta W_{n},T_{\Delta t})\Big),\nonumber
  \end{align*}
  then
  \begin{align*}
     & \mathbf{E}|\tilde{\psi}(t_{n+1})-\phi^{n+1}|_{\alpha}^{2}\\
    &  =\mathbf{E}|\tilde{\psi}_{t_{n},\tilde{\psi}(t_{n})}(t_{n+1})-\tilde{\psi}_{t_{n},\phi^{n}}(t_{n+1})|_{\alpha}^{2}
     +\mathbf{E}|\Upsilon(\tilde{\psi}_{t_{n},\phi^{n}}(r))-\Gamma(\phi^{n},\phi^{n+1},\Delta t,\Delta W_{n},T_{\Delta t})|_{\alpha}^{2}\\
  &  \quad  +2\mathbf{E}\langle\tilde{\psi}_{t_{n},\tilde{\psi}(t_{n})}(t_{n+1})-\tilde{\psi}_{t_{n},\phi^{n}}(t_{n+1}),
\Upsilon(\tilde{\psi}_{t_{n},\phi^{n}}(r))-\Gamma(\phi^{n},\phi^{n+1},\Delta t,\Delta W_{n},T_{\Delta t})\rangle_{\alpha}.\nonumber
  \end{align*}
  From Proposition \ref{lemma3}, we have
  \begin{equation}
    \mathbf{E}|\tilde{\psi}_{t_{n},\tilde{\psi}(t_{n})}(t_{n+1})-\tilde{\psi}_{t_{n},\phi^{n}}(t_{n+1})|_{\alpha}^{2}
    \leq (1+K\Delta t)\mathbf{E}|\tilde{\psi}(t_{n})-\phi^{n}|_{\alpha}^{2},\nonumber
  \end{equation}
  and from condition (\ref{cond2}), we have
  \begin{equation}
    \mathbf{E}|\Upsilon(\tilde{\psi}_{t_{n},\phi^{n}}(r))-\Gamma(\phi^{n},\phi^{n+1},h,\Delta W_{n},T_{\Delta t})|_{\alpha}^{2}
    \leq K\Delta t^{2p+1}.\nonumber
  \end{equation}
 From Proposition \ref{lemma3}, we may split $\tilde{\psi}_{t_{n},\tilde{\psi}(t_{n})}(t_{n+1})-\tilde{\psi}_{t_{n},\phi^{n}}(t_{n+1})$ as $\hat{S}_{\Delta t}(\tilde{\psi}(t_{n})-\phi^{n})$ and $Z$,
  and use the trick of conditional expectation with respect to $\mathcal{F}_{t_{n}}$, then get
  \begin{align*}
     & \mathbf{E}\langle\tilde{\psi}_{t_{n},\tilde{\psi}(t_{n})}(t_{n+1})-\tilde{\psi}_{t_{n},\phi^{n}}(t_{n+1}),
\Upsilon(\tilde{\psi}_{t_{n},\phi^{n}}(r))-\Gamma(\phi^{n},\phi^{n+1},\Delta t,\Delta W_{n},T_{\Delta t})\rangle_{\alpha}\\
  &\leq K\Delta t\mathbf{E}|\tilde{\psi}(t_{n})-\phi^{n}|_{\alpha}^{2}+K\Delta t^{2p+1}.\nonumber
  \end{align*}
  Therefore
  $
    \mathbf{E}|\tilde{\psi}(t_{n+1})-\phi^{n+1}|_{\alpha}^{2}\leq(1+K\Delta t)\mathbf{E}|\tilde{\psi}(t_{n})-\phi^{n}|_{\alpha}^{2}+K\Delta t^{2p+1}.
 $
  Hence by Gronwall's lemma, we obtain
  $
    \mathbf{E}|\tilde{\psi}(t_{n+1})-\phi^{n+1}|_{\alpha}^{2} \leq K\Delta t^{2p}.
 $

  Next we consider the estimate between $\psi(t_{n+1})$ and $\tilde{\psi}(t_{n+1})$, since
  \begin{align*}
    \psi(t_{n+1})=&S(t_{n+1})\varphi+\mathbf{i}\sum_{k=0}^{n}\int_{t_{k}}^{t_{k+1}}S(t_{n+1}-r)F(\psi(r))dr\\
   & -\mathbf{i}\sum_{k=0}^{n}\int_{t_{k}}^{t_{k+1}}S(t_{n+1}-r)G(\psi(r))dW(r),
  \end{align*}
  and using \eqref{tilde psi} recurrently
  \begin{align*}
    \tilde{\psi}(t_{n+1})
   &=\hat{S}_{\Delta t}^{n+1}\varphi+\mathbf{i}\sum_{k=0}^{n}\int_{t_{k}}^{t_{k+1}}\hat{S}_{\Delta t}^{n-k}S(t_{k+1}-r)F(\tilde{\psi}(r))dr\\
   &\quad -\mathbf{i}\sum_{k=0}^{n}\int_{t_{k}}^{t_{k+1}}\hat{S}_{\Delta t}^{n-k}S(t_{k+1}-r)G(\tilde{\psi}(r))dW(r)\nonumber
  \end{align*}
  then
  \begin{align*}
      &\psi(t_{n+1})-\tilde{\psi}(t_{n+1})=(S(t_{n+1})-\hat{S}_{\Delta t}^{n+1})\varphi\\
     &+\mathbf{i}\sum_{k=0}^{n}\int_{t_{k}}^{t_{k+1}}\Big(S(t_{n+1}-r)F(\psi(r))-\hat{S}_{\Delta t}^{n-k}S(t_{k+1}-r)F(\tilde{\psi}(r))\Big)dr\\
      &-\mathbf{i}\sum_{k=0}^{n}\int_{t_{k}}^{t_{k+1}}\Big(S(t_{n+1}-r)G(\psi(r))-\hat{S}_{\Delta t}^{n-k}S(t_{k+1}-r)G(\tilde{\psi}(r))\Big)dW(r)\\
      :=& {\mathcal T}_{1}+{\mathcal T}_{2}+{\mathcal T}_{3}.\nonumber
  \end{align*}
  The estimation of ${\mathcal T}_1$ follows from Hypothesis \ref{hypo3}, that is $|{\mathcal T}_1|_{\alpha}\leq |S(t_{n+1})-\hat{S}_{\Delta t}^{n+1}|_{\mathcal{L}({\mathbb H}^{\beta},{\mathbb H}^{\alpha})}|\varphi|_{\beta}\leq K\Delta t^{q}$.
  By the Lipschitz property of $F$, we have
  \begin{align*}
    {\mathbf E}|{\mathcal T}_2|_{\alpha}^2\leq& K\sum_{k=0}^{n}\int_{t_k}^{t_{k+1}}{\mathbf E}\Big|S(t_{n+1}-r)F(\psi(r))-\hat{S}_{\Delta t}^{n-k}S(t_{k+1}-r)F(\tilde{\psi}(r))\Big|_{\alpha}^{2}dr \\
    \leq & K\sum_{k=0}^{n}\int_{t_k}^{t_{k+1}}{\mathbf E}\Big|S(t_{n+1}-r)\big(F(\psi(r))-F(\tilde{\psi}(r))\big)\Big|_{\alpha}^{2}dr \\
    &+K\sum_{k=0}^{n}\int_{t_k}^{t_{k+1}}{\mathbf E}\Big|\big(S(t_{n+1}-t_{k+1})-\hat{S}_{\Delta t}^{n-k}\big)S(t_{k+1}-r)F(\tilde{\psi}(r))\Big|_{\alpha}^{2}dr \\
    \leq & K\sum_{k=0}^{n}\int_{t_k}^{t_{k+1}}{\mathbf E}\Big|\psi(r)-\tilde{\psi}(r)\Big|_{\alpha}^{2}dr+K\Delta t^{2q}\sum_{k=0}^{n}\int_{t_k}^{t_{k+1}}{\mathbf E}\Big|F(\tilde{\psi}(r))\Big|_{\beta}^{2}dr\\
    \leq & K\Delta t^{2q}+ K\int_{t_0}^{t_{n+1}}{\mathbf E}\Big|\psi(r)-\tilde{\psi}(r)\Big|_{\alpha}^{2}dr.
  \end{align*}
  For term ${\mathcal T}_3$, we use the property \eqref{eq5} of stochastic integral to get
  \begin{align*}
    {\mathbf E}|{\mathcal T}_3|_{\alpha}^2=& \sum_{k=0}^{n}\int_{t_k}^{t_{k+1}}{\mathbf E}\Big|\Big(S(t_{n+1}-r)G(\psi(r))-\hat{S}_{\Delta t}^{n-k}S(t_{k+1}-r)G(\tilde{\psi}(r))\Big){\mathbf Q}^{\frac12}\Big|_{HS(U,{\mathbb H}^{\alpha})}^{2}dr \\
    \leq & K\Delta t^{2q}+ K\int_{t_0}^{t_{n+1}}{\mathbf E}\Big|\psi(r)-\tilde{\psi}(r)\Big|_{\alpha}^{2}dr.
  \end{align*}
  Hence combining them together and by Gronwall's lemma, we have
  \begin{equation*}
    \mathbf{E}|\psi(t_{n+1})-\tilde{\psi}(t_{n+1})|_{\alpha}^{2}\leq K\Delta t^{2q}.
  \end{equation*}
  At last, we estimate the mean-square error between $\psi(t_{n+1})$ and $\phi^{n+1}$ by triangle inequality
  \begin{align*}
      \mathbf{E}|\psi(t_{n+1})-\phi^{n+1}|_{\alpha}^{2} &\leq 2\mathbf{E}|\psi(t_{n+1})-\tilde{\psi}(t_{n+1})|_{\alpha}^{2}
      +2\mathbf{E}|\tilde{\psi}(t_{n+1})-\phi^{n+1}|_{\alpha}^{2}\\
      &\leq K\Delta t^{2q}+K\Delta t^{2p}
      \leq K\Delta t^{\min\{2q,2p\}}.\nonumber
  \end{align*}
  The proof of the convergence theorem is completed.
\end{proof}
\begin{remark}
  Consider the stochastic Schr\"{o}dinger equation in Stratonovich sense
  \vspace{-0.5em}
  \begin{equation*}
      \mathbf{i}d\psi+(A\psi+F(\psi))dt=G(\psi)\circ dW.
  \end{equation*}
  It is well known that this equation is equivalent to the following equation in the sense of It\^{o}
  \begin{equation}
      \mathbf{i}d\psi+(A\psi+\tilde{F}(\psi))dt=G(\psi)dW\nonumber
  \end{equation}
  where $\aleph_{{\mathbf Q}}=\sum_{i\in {\mathbb N}}({\mathbf Q}^{\frac12}e_{i}(x))^{2}$
  and $\tilde{F}(\psi)=F(\psi)+\frac{{\mathbf i}}{2}G^{\prime}(\psi)G(\psi)\aleph_{{\mathbf Q}}$.
  We assume that the coefficients $\tilde{F}$ and $G$ of the equation satisfy the hypothesis \ref{hypo1} and \ref{hypo2}, then it is not difficult to understand that Theorem \ref{convergence theorem}
remains true for equation understood in the sense of Stratonovich.
\end{remark}

\subsection{Mean-square convergence order of the midpoint scheme}\label{section}
Here we use the convergence theorem (Theorem \ref{convergence theorem}) to obtain the mean-square convergence order
of the symplectic semi-discrete scheme (\ref{symp scheme}) with $F$ satisfying the following condition:
$F$ is twice Fr\'echet differentiable
with bounded derivatives. Otherwise, truncation strategy could be employed as in \cite{Bouard2006,Liu}.

In order to obtain the mean-square convergence order, we require three extra regularities of the solution for equation \eqref{eq1}, namely, we set $\beta=\alpha+3$. In this case, we have the following estimations to operators \eqref{hats} (\cite{Bouard2006}), which are useful in the calculation below:
\begin{align}\label{operator}
  &|\hat{S}_{\Delta t}|_{\mathcal{L}(U_c,U_c)}\leq 1,\qquad  |T_{\Delta t}|_{\mathcal{L}(U_c,U_c)}\leq 1,\qquad
  |S(t_n)-\hat{S}_{\Delta t}^{n}|_{{\mathcal L}({\mathbb H}^{\alpha+3}, {\mathbb H}^{\alpha})}\leq K\Delta t,\\
 & |\hat{S}_{\Delta t}-I|_{{\mathcal L}({\mathbb H}^{\alpha+2}, {\mathbb H}^{\alpha})}\leq K\Delta t,\qquad
  |T_{\Delta t}-I|_{{\mathcal L}({\mathbb H}^{\alpha+2}, {\mathbb H}^{\alpha})}\leq K\Delta t.\nonumber
\end{align}
The last inequality in the first line of \eqref{operator} means that $q=1$.

Let $\xi_{i}$, $i\in{\mathbb N}$ be $\mathcal{N}(0,1)$-distributed random variable. Due to the implicity in diffusion and the possibility of the noise could become unbounded for any arbitrary small time step size,
 as in \cite{Milstein1995} we truncate the noise $\Delta W_{n}=\sum_{i\in{\mathbb N}} \Delta \beta_{i}{\bf Q}^{\frac12}e_{i}
 =\sqrt{\Delta t}\sum_{i\in{\mathbb N}} \xi_{i}{\mathbf Q}^{\frac12}e_{i}$ in \eqref{symp scheme}
 by another random variable \[\Delta \bar{W}_{n}=\sqrt{\Delta t}\sum_{i\in{\mathbb N}} \zeta_{i}{\mathbf Q}^{\frac12}e_{i}.\]
For $A_{\Delta t}=\sqrt{2k|\ln \Delta t|}$ ($k\geq 1$), let
\begin{equation*}
  \zeta_{i}=
  \begin{cases}
    \xi_{i}\quad  |\xi_{i}|\leq A_{\Delta t}, \\
    A_{\Delta t}\quad \xi_{i}> A_{\Delta t},\\
    -A_{\Delta t}\quad  \xi_{i}< -A_{\Delta t}.
  \end{cases}
\end{equation*}
Here, the role of parameter $k$ is such that ${\bf E}(\zeta_i-\xi_i)^2\leq \Delta t^k$ and
${\bf E}\big(\zeta_i^2-\xi_i^2\big)\leq (1+2A_{\Delta t})\Delta t^k$.
Thus for ${\mathbf Q}^{\frac12}\in HS(U, {\mathbb H}^{\alpha})$, we have
\begin{equation}\label{noise}
\begin{split}
&{\mathbf E}|\Delta \bar{W}_{n}-\Delta W_{n}|^2_{\alpha}\leq K\Delta t^{k+1},\\
 & {\mathbf E}|(\Delta\bar{W}_{n})^2-(\Delta W_{n})^2|^{2}_{\alpha}\leq K\Delta t^{k+2}(1+A_{\Delta t}+A_{\Delta t}^2)
  \leq K\Delta t^{k+1},
\end{split}
\end{equation}
where we use the property $\Delta t(1+A_{\Delta t}+(A_{\Delta t})^2)\leq 1$ for sufficient small $\Delta t$.

For the following analysis, we require $k=2$ and have the following results.
\begin{theorem}
  The midpoint scheme is of order $1$ in the mean-square convergence sense, i.e., for $\alpha>\frac{d}{2}$ with $d$ being the dimension of the problem and for sufficient small $\Delta t$, we have
  \[\Big({\mathbf E}|\psi(t_{n})-\phi^{n}|_{\alpha}^2\Big)^{\frac12}\leq K\Delta t \qquad \forall n=0,1,\cdots, N.\]
\end{theorem}
\begin{proof}
First of all, we show that the numerical solution \eqref{symp scheme} with truncated noise has bounded moments. In fact,
\begin{align}\label{eq6}
  {\mathbf E}|\phi^{n+1}|_{\alpha}^2=&{\mathbf E}\Big|\hat{S}_{\Delta t}^{n+1}\varphi+{\mathbf i}\Delta t\sum_{k=0}^{n}\hat{S}_{\Delta t}^{n-k}T_{\Delta t}F(t_{k+\frac12},\phi^{k+\frac12})\\
  &\quad-{\mathbf i}\varepsilon \sum_{k=0}^{n}\hat{S}_{\Delta t}^{n-k}T_{\Delta t}\Big(\frac12(I+\hat{S}_{\Delta t})\phi^{k}+\frac12(\phi^{k+1}-\hat{S}_{\Delta t}\phi^{k})\Big)\Delta\bar{W}_{k}\Big|_{\alpha}^2\nonumber\\
  \leq & K{\mathbf E}|\varphi|_{\alpha}^2+K\Delta t\sum_{k=0}^{n}{\mathbf E}|\phi^{k+\frac12}|_{\alpha}^{2}+Kn\sum_{k=0}^{n}\big|(\phi^{k+1}-\hat{S}_{\Delta t}\phi^{k})\Delta\bar{W}_{k}\big|_{\alpha}^2.\nonumber
\end{align}
Since
\begin{align*}
  &{\mathbf E}\big|(\phi^{k+1}-\hat{S}_{\Delta t}\phi^{k})\Delta\bar{W}_{k}\big|_{\alpha}^2\\
  &={\mathbf E}\Big|{\mathbf i}\Delta tT_{\Delta t}F(t_{k+\frac12},\phi^{k+\frac12})\Delta\bar{W}_{k}-{\mathbf i}\varepsilon T_{\Delta t}\phi^{k}(\Delta\bar{W}_{k})^2-\frac{{\mathbf i}}{2}\varepsilon T_{\Delta t}\big((\phi^{k+1}-\hat{S}_{\Delta t}\phi^{k})\Delta\bar{W}_{k}\big)\Delta\bar{W}_{k}\Big|_{\alpha}^2\\
  &\leq K\Delta t^2\Delta t(A_{\Delta t})^2{\mathbf E}|\phi^{k+\frac12}|_{\alpha}^2+K\Delta t^2{\mathbf E}|\phi^{k}|_{\alpha}^2
  +K_{0}\Delta t (A_{\Delta t})^2{\mathbf E}\big|(\phi^{k+1}-\hat{S}_{\Delta t}\phi^{k})\Delta\bar{W}_{k}\big|_{\alpha}^2.
\end{align*}
There exists $\Delta t^{*}$ such that $\forall\, \Delta t<\Delta t^{*}$ one has $K_{0}\Delta t (A_{\Delta t})^2\leq \frac12$, therefore
\begin{equation}\label{eq7}
  {\mathbf E}\big|(\phi^{k+1}-\hat{S}_{\Delta t}\phi^{k})\Delta\bar{W}_{k}\big|_{\alpha}^2\leq K\Delta t^2{\mathbf E}|\phi^{k+\frac12}|_{\alpha}^2.
\end{equation}
Substituting \eqref{eq7} into the right-hand side of \eqref{eq6}, and by Gronwall inequality one arrives at
$${\mathbf E}|\phi^{n}|_{\alpha}^2\leq K(1+{\mathbf E}|\varphi|_{\alpha}^2)\leq K.$$
From (\ref{mild solution of S equation}), we know that
\begin{align*}
  \Upsilon(\tilde{\psi}_{t_{n},\phi^n}(r))=&\mathbf{i}\int_{t_{n}}^{t_{n+1}}S(t_{n+1}-r)F(r,\tilde{\psi}_{t_{n},\phi^n}(r))dr
  -\mathbf{i}\varepsilon \int_{t_{n}}^{t_{n+1}}S(t_{n+1}-r)\tilde{\psi}_{t_{n},\phi^n}(r)dW(r)\\
  &-\frac{\varepsilon^{2}}{2}\int_{t_{n}}^{t_{n+1}}S(t_{n+1}-r)\tilde{\psi}_{t_{n},\phi^n}(r)\aleph_{\mathbf Q}dr,\nonumber
\end{align*}
and from (\ref{symp scheme}) with truncated noise we have
\begin{align*}
  \Gamma(\phi^n,\phi^{n+1},\Delta t,\Delta \bar{W}_{n},T_{\Delta t})
  =&\mathbf{i}\Delta tT_{\Delta t}F(t_{n+\frac12},\phi^{n+\frac12})
-\mathbf{i}\varepsilon T_{\Delta t}\phi^n\Delta \bar{W}_{n}-\frac{{\mathbf i}\varepsilon}{2}T_{\Delta t}(\phi^{n+1}-\phi^n)\Delta \bar{W}_{n}.
\end{align*}
Then
\begin{align*}
    \mathbf{E}^{t_n}&\Big(\Upsilon(\tilde{\psi}_{t_{n},\phi^n}(r))-\Gamma(\phi^n,\phi^{n+1},\Delta t,\Delta \bar{W}_{n},T_{\Delta t})\Big)\\
    =&\mathbf{E}^{t_n}\Big(\mathbf{i}\int_{t_{n}}^{t_{n+1}}S(t_{n+1}-r)F(r,\tilde{\psi}_{t_{n},\phi^n}(r))dr-\mathbf{i}\Delta tT_{\Delta t}F(t_{n+\frac12},\phi^{n+\frac12})\Big)\\
   &+\mathbf{E}^{t_n}\Big(\frac{\mathbf{i}\varepsilon}{2}T_{\Delta t}(\phi^{n+1}-\phi^n)\Delta \bar{W}_{n}-\frac{\varepsilon^{2}}{2}\int_{t_{n}}^{t_{n+1}}S(t_{n+1}-r)\tilde{\psi}_{t_{n},\phi^n}(r)\aleph_{{\mathbf Q}}dr\Big)\\
    :=&\mathcal{A}+\mathcal{B}.\nonumber
\end{align*}
We split $\mathcal{A}$ further
\begin{align}
    \mathcal{A}=&\mathbf{i}\mathbf{E}^{t_n}\int_{t_{n}}^{t_{n+1}}(S(t_{n+1}-r)-T_{\Delta t})F(r,\tilde{\psi}_{t_{n},\phi^n}(r))dr\\
    &+\mathbf{i}\mathbf{E}^{t_n}\int_{t_{n}}^{t_{n+1}}T_{\Delta t}(F(r,\tilde{\psi}_{t_{n},\phi^n}(r))-F(t_{n+\frac12},\phi^n))dr\nonumber\\
    &-\mathbf{i}\Delta tT_{\Delta t}\mathbf{E}^{t_n}(F(t_{n+\frac12},\phi^{n+\frac12})-F(t_{n+\frac12},\phi^n))\nonumber\\
    :=&\mathcal{A}^{1}+\mathcal{A}^{2}+\mathcal{A}^{3}.\nonumber
\end{align}
From \eqref{operator}, we have $
  {\mathbf E}|\mathcal{A}^{1}|_{\alpha}^2\leq K\Delta t^{4}.
$
In addition, via the boundedness of derivatives of $F$ with respect to $\psi$,
we write
\begin{align}
    F(r,\tilde{\psi}_{t_{n},\phi^n}(r))-F(t_{n+\frac12},\phi^n)
    &=\Big(F(r,\tilde{\psi}_{t_{n},\phi^n}(r))-F(r,\phi^n)\Big)+\Big(F(r,\phi^n)-F(t_{n+\frac12},\phi^n)\Big)\nonumber\\
    &=\frac{\partial F}{\partial \psi}(r,\phi^n)(\tilde{\psi}_{t_{n},\phi^n}(r)-\phi^n)+\rho_{1}+\rho_{2},\nonumber
\end{align}
where $|\rho_{1}|_{\alpha}\leq K|\tilde{\psi}_{t_{n},\phi^n}(r)-\phi^n|_{\alpha}^{2}\text{ and }{\mathbf E}|\rho_{2}|^{2}_{\alpha}\leq K\Delta t^2.$
It is not difficult to obtain the estimate
\[\int_{t_{n}}^{t_{n+1}}{\mathbf E}|\mathbf{E}^{t_n}(\tilde{\psi}_{t_{n},\phi^n}(r)-\phi^n)|^2_{\alpha}dr
\leq K\Delta t^{3}.\]
The above implies that ${\mathbf E}|\mathcal{A}^{2}|_{\alpha}^2\leq K\Delta t^{4}.$
Similarly as the estimate of $\mathcal{A}^{2}$, one has
$
  {\mathbf E}|\mathcal{A}^{3}|_{\alpha}^2\leq K\Delta t^{4}.
$
The estimate of $\mathcal{B}$ is more technical. First inserting the expression of $\phi^{n+1}-\phi^n$ into $\mathcal{B}$, we have
\begin{align*}
  \mathcal{B}=&\mathbf{E}^{t_n}\Big[\frac{\mathbf{i}\varepsilon}{2}T_{\Delta t}\Big((\hat{S}_{\Delta t}-I)\phi^n+\mathbf{i}\Delta tT_{\Delta t}F(t_{n+\frac12},\phi^{n+\frac12})-\mathbf{i}\varepsilon T_{\Delta t}\phi^n\Delta \bar{W}_{n}\\
  &-\frac{\mathbf{i}\varepsilon}{2}T_{\Delta t}(\phi^{n+1}-\phi^n)\Delta \bar{W}_{n}\Big)\Delta \bar{W}_{n}\nonumber
  -\frac{\varepsilon^{2}}{2}\int_{t_{n}}^{t_{n+1}}S(t_{n+1}-r)\tilde{\psi}_{t_{n},\phi^n}(r)\aleph_{{\mathbf Q}}dr\Big]\nonumber\\
    =&-\frac{\varepsilon}{2}\Delta t T_{\Delta t}^{2}\mathbf{E}^{t_n}\Big(F(t_{n+\frac12},\phi^{n+\frac12})\Delta \bar{W}_{n}\Big)
  +\frac{\varepsilon^{2}}{4}T_{\Delta t}^{2}\mathbf{E}^{t_n}\Big((\phi^{n+1}-\phi^n)(\Delta \bar{W}_{n})^{2})\Big)\nonumber\\
   & +\frac{\varepsilon^{2}}{2}\Big[T_{\Delta t}^{2}\phi^n{\mathbf E}^{t_n}(\Delta\bar{W}_n)^2-\int_{t_{n}}^{t_{n+1}}\mathbf{E}^{t_n}
    S(t_{n+1}-r)\tilde{\psi}_{t_{n},\phi^n}(r)\aleph_{{\mathbf Q}}dr\Big]\nonumber\\
    :=& \mathcal{B}^{1}+\mathcal{B}^{2}+\mathcal{B}^{3}.\nonumber
\end{align*}
We write $F(t_{n+\frac12},\phi^{n+\frac12})=F(t_{n+\frac12},\phi^{n})+\frac12\frac{\partial F}{\partial \psi}(t_{n+\frac12},\phi^{n})(\phi^{n+1}-\phi^{n})+\rho_1$ and use \eqref{eq7} to obtain
$
  {\mathbf E}|\mathcal{B}^{1}|^2_{\alpha}\leq K\Delta t^{4}.
$
Inserting the expression of $\phi^{n+1}-\phi^n$ into the term $\mathcal{B}^{2}$, similarly, one can obtain that
$
  \mathbf{E}|\mathcal{B}^{2}|_{\alpha}^{2}\leq K\Delta t^{4}.\nonumber
$
We split $\mathcal{B}^{3}$ further
\begin{align*}
    \mathcal{B}^{3}
    =&\frac{\varepsilon^{2}}{2}T_{\Delta t}\phi^n{\mathbf E}^{t_n}\Big((\Delta \bar{W}_n)^2-(\Delta W_{n})^2\Big)
    +\frac{\varepsilon^{2}}{2}{\mathbf E}^{t_n}\int_{t_n}^{t_{n+1}}\Big(T_{\Delta t}^2\phi^n-S(t_{n+1}-r)\tilde{\psi}_{t_n,\phi^n}(r)\Big)\aleph_{\mathbf Q}dr
\end{align*}
where the first term could be bounded by the second inequality in \eqref{noise}, and the second term on the above equality could be estimated similarly by inserting the expression of $\tilde{\psi}_{t_n,\phi^n}(r)-\phi^{n}$ into it, which lead to
$
  \mathbf{E}|\mathcal{B}^{3}|_{\alpha}^{2}\leq K\Delta t^{4}.
$
Therefore, we have
\[{\mathbf E}\Big|\mathbf{E}^{t_n}\Big(\Upsilon(\tilde{\psi}_{t_{n},\phi^n}(r))-\Gamma(\phi^n,\phi^{n+1},\Delta t,\Delta \bar{W}_{n},T_{\Delta t})\Big)\Big|_{\alpha}^2 \leq K\Delta t^{4}.\]
Next, let's compute the value of $\mathbf{E}|\Upsilon(\tilde{\psi}_{t_{n},\phi^n}(r))-\Gamma(\phi^n,\phi^{n+1},\Delta t,\Delta \bar{W}_{n},T_{\Delta t})|_{\alpha}^{2}$.
\begin{align*}
    &\Upsilon(\tilde{\psi}_{t_{n},\phi^n}(r))-\Gamma(\phi^n,\phi^{n+1},\Delta t,\Delta \bar{W}_{n},T_{\Delta t})\\
    &=\mathbf{i}\int_{t_{n}}^{t_{n+1}}\Big(S(t_{n+1}-r)F(r,\tilde{\psi}_{t_{n},\phi^n}(r))-T_{\Delta t}F(t_{n+\frac12},\phi^{n+\frac12})\Big)dr\\
   &-\mathbf{i}\varepsilon \int_{t_{n}}^{t_{n+1}}(S(t_{n+1}-r)-T_{\Delta t})\tilde{\psi}_{t_{n},\phi^n}(r)dW(r)\\
    &-\Big[\mathbf{i}\varepsilon \int_{t_{n}}^{t_{n+1}}T_{\Delta t}(\tilde{\psi}_{t_{n},\phi^n}(r)-\phi^n)dW(r)
    +\frac{\varepsilon^{2}}{2}\int_{t_{n}}^{t_{n+1}}S(t_{n+1}-r)\tilde{\psi}_{t_{n},\phi^n}(r)\aleph_{{\mathbf Q}}dr\\
    &\quad\quad-\frac{\mathbf{i}\varepsilon}{2}T_{\Delta t}(\phi^{n+1}-\phi^n)\Delta \bar{W}_{n}+{\mathbf i}\varepsilon T_{\Delta t}\phi^{n}\big(\Delta W_n-\Delta\bar{W}_n\big)\Big]\\
   :=&\mathcal{L}+\mathcal{M}-\mathcal{N}.
\end{align*}
The estimation of ${\mathcal L}$ is similar as that of term ${\mathcal A}$, that is $\mathbf{E}|\mathcal{L}|_{\alpha}^{2}\leq K\Delta t^{3}.$
By \eqref{eq5} and \eqref{operator}, we know that $\mathbf{E}|\mathcal{M}|_{\alpha}^{2}\leq K\Delta t^{3}.$
Inserting the expression of $\tilde{\psi}(r)-\phi^n$ and $\phi^{n+1}-\phi^n$ into $\mathcal{N}$ and splitting $\mathcal{N}$ further
\begin{align*}
    \mathcal{N}=&\mathbf{i}\varepsilon\int_{t_{n}}^{t_{n+1}}T_{\Delta t}\Big((\hat{S}_{r-t_{n}}-I)\phi^n+\mathbf{i}\int_{t_{n}}^{r}S(r-\rho)\tilde{F}(\tilde{\psi}_{t_{n},\phi^n}(\rho))d\rho\Big)dW(r)\\
    &-\frac{\mathbf{i}\varepsilon}{2}T_{\Delta t}\Big((\hat{S}_{\Delta t}-I)\phi^n+\mathbf{i}\Delta t T_{\Delta t}F(t_{n+\frac12},\phi^{n+\frac12})\Big)\Delta \bar{W}_{n}+{\mathbf i}\varepsilon T_{\Delta t}\phi^{n}\big(\Delta W_n-\Delta\bar{W}_n\big)\\
    &+\varepsilon^{2}\int_{t_{n}}^{t_{n+1}}\int_{t_{n}}^{r}T_{\Delta t}S(r-\rho)\tilde{\psi}_{t_{n},\phi^n}(\rho)dW(\rho)dW(r)
    +\frac{\varepsilon^{2}}{2}\int_{t_{n}}^{t_{n+1}}S(t_{n+1}-r)\tilde{\psi}_{t_{n},\phi^n}(r)\aleph_{{\mathbf Q}}dr\\
    &-\frac{\varepsilon^2}{2}T_{\Delta t}^2\phi^{n}(\Delta W_n)^2+\frac{\varepsilon^2}{2}T_{\Delta t}^2\phi^{n}\Big((\Delta W_n)^2-(\Delta \bar{W}_n)^2\Big)-\frac{\varepsilon^2}{4}T_{\Delta t}^2(\phi^{n+1}-\phi^{n})(\Delta\bar{W}_n)^2
\end{align*}
Special attention should be paid to the third line and the first term in the last line of the above equality, since other terms could be estimated as before.
In fact, $$
\frac{\varepsilon^2}{2}T_{\Delta t}^2\phi^{n}(\Delta W_n)^2=\varepsilon^2\int_{t_n}^{t_{n+1}}\int_{t_n}^{r}T_{\Delta t}^2\phi^{n}dW(\rho)dW(r)+\frac{\varepsilon^2}{2}T_{\Delta t}^2\phi^{n}\aleph_{\mathbf Q}\Delta t
$$
leads to
$
  \mathbf{E}|\mathcal{N}|_{\alpha}^{2}\leq K\Delta t^{3}.\nonumber
$
That is \[\Big( \mathbf{E}|\Upsilon(\tilde{\psi}_{t_{n},\phi^n}(r))-\Gamma(\phi^n,\phi^{n+1},\Delta t,\Delta \bar{W}_{n},T_{\Delta t})|_{\alpha}^{2}\Big)^{\frac12}\leq K\Delta t^{\frac32}.\]
Therefore the mean-square order of midpoint scheme is of $1$ according to Theorem \ref{convergence theorem}.
\end{proof}
\begin{remark}
For stochastic Schr\"odinger equation in Stratonovich sense, the mean-square convergence order of the semi-discrete midpoint scheme is 1 under appropriate assumptions. It seems like the same as the case of stochastic ordinary differential equations. However, the convergence order here depends on the values of $p$ and $q$, where $p$ and $q$ are from estimations of the one-step deviation between exact solution and numerical solution. As we take parameters $\beta=\alpha+3$, which require three more regularity conditions on the solution $\psi$, the estimation of operators is $|S(t_{k})-\hat{S}_{\Delta t}^{k}|_{\mathcal{L}(H^{\beta},H^{\alpha})}\leq K\Delta t$ with $q=1$. Together with $p=1$, we have that the convergence order of the semi-discrete midpoint scheme is 1. If we put less regularity on the solution, which means $q<1$, then the convergence order is less than 1. For the general stochastic Runge-Kutta methods in temporal direction, the mean-square convergence order is an open problem.
\end{remark}
\section{Numerical experiments}
In the Section \ref{section}, we showed the  convergence in the mean-square sense of midpoint scheme \eqref{symp scheme} with spatially regular noise and under certain assumptions the order is of $1$. The following example is chosen to study convergence order computationally of the midpoint scheme \eqref{symp scheme} to solve the stochastic cubic Schr\"odinger equation on
$[-1,1]\times[0,T]$
\begin{align}\label{eq501}
 & \mathbf{i}d\psi(x,t)+\Big(\partial_{xx}\psi+|\psi|^{2}\psi(x,t)\Big)dt=\varepsilon\psi(x,t)\circ dW(t),\\
  &\psi(x,0)=\sin(\pi x).\nonumber
\end{align}
Let $T=\frac14$. For integer $M$, and $\{\beta_{\ell};1\leq \ell\leq M\}$ a
family of independent ${\mathbb R}$-valued Wiener processes, consider the real-valued Wiener process $W(t)=\sum_{\ell=1}^{M}\frac{1}{\ell}\sin(\pi\ell x)\beta_{\ell}(t)$, and $\varepsilon=\sqrt{2}$ in \eqref{eq501}.
 We use the midpoint scheme \eqref{symp scheme} in the temporal direction and finite difference in the spatial direction for the numerical approximation. Let $I_{\Delta t}=\{t_n;0\leq n\leq N\}$ be the uniform
discretization of $[0,\,T]$ of size $\Delta t>0$, and apply the uniform discretization of $[-1,1]$ of size $\Delta x=\frac{1}{256}$. The reference values (for
Figure 1b) are generated for the smallest mesh size $\Delta t=2^{-14}$. And
$500$ realizations are chosen to approximate the expectations.
\begin{figure}[htbp]
  \centering
  \begin{minipage}[b]{5.2 cm}\label{Fig1a}
   \begin{center} a) $\varepsilon = 0$  \end{center}
   \centering
    \includegraphics[width=1.078\textwidth]{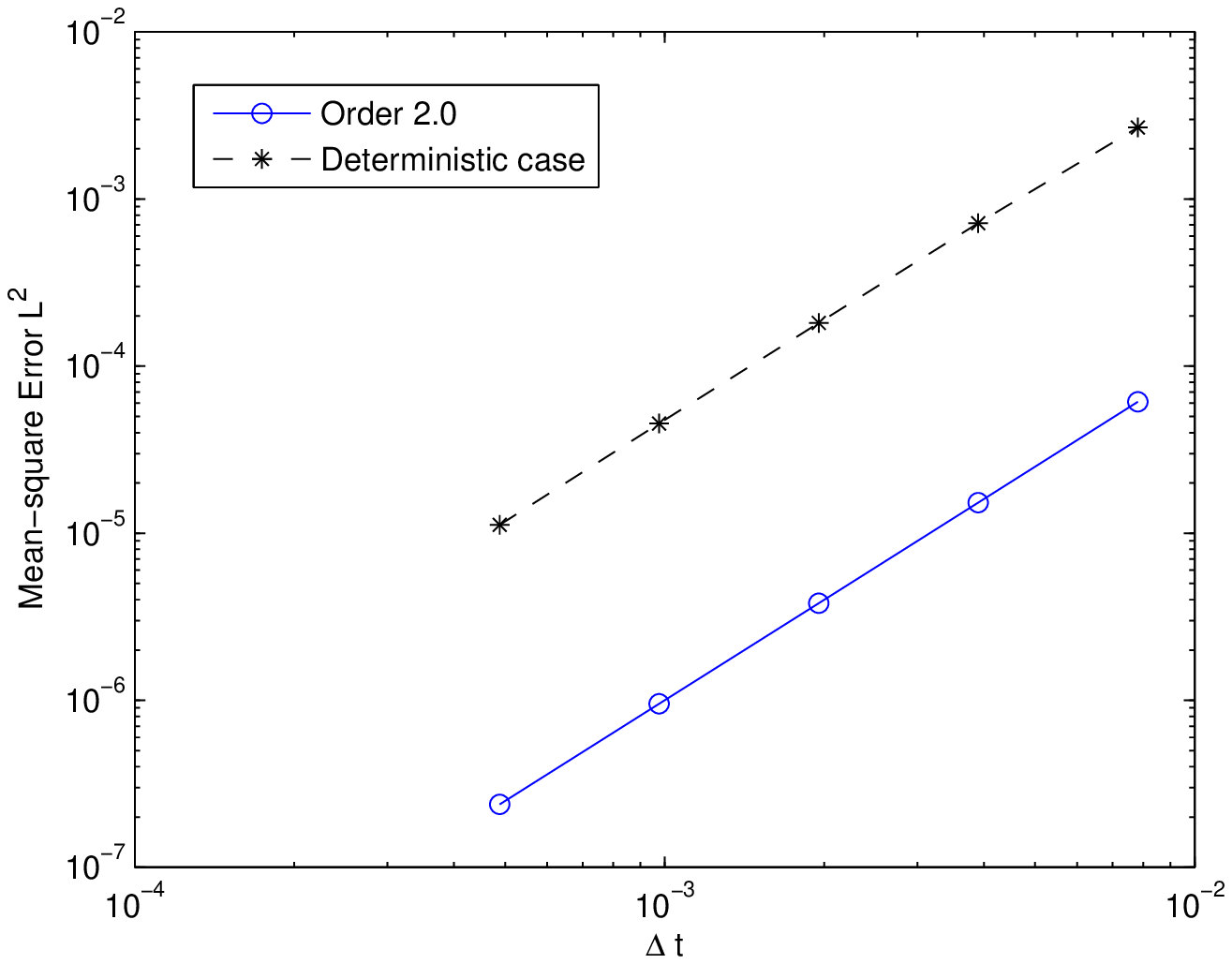}
    \end{minipage}
      \begin{minipage}[b]{5.2 cm}\label{Fig1b}
     \begin{center} b) $\varepsilon = \sqrt{2}$    \end{center}
     \centering
    \includegraphics[width=1.078\textwidth]{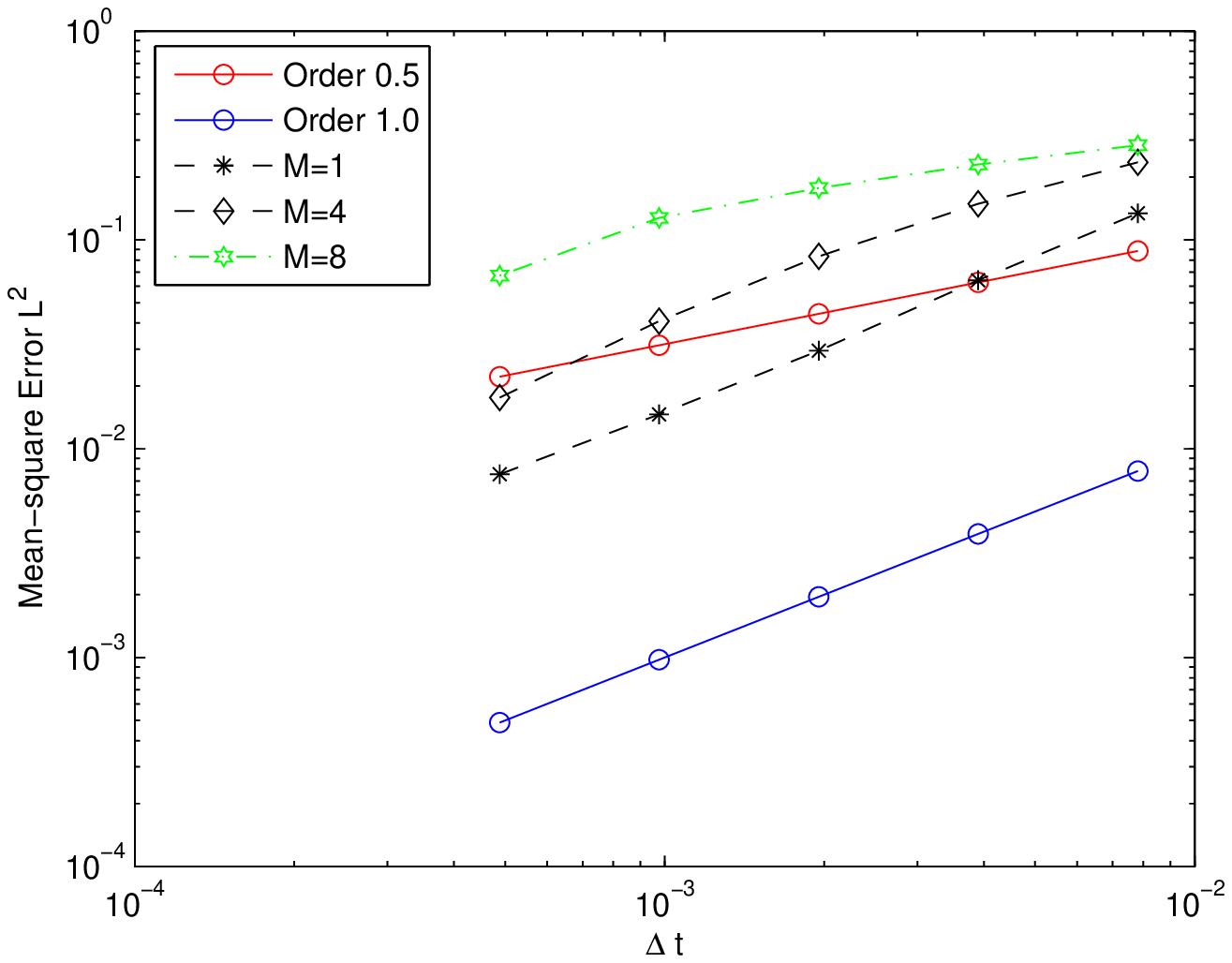}
  \end{minipage}
  \caption{ a) Rates of convergence for the deterministic case in the norm $\Vert \psi(T) - \phi^{N}\Vert_{{\mathbb L}^2}$ ($T = \frac{1}{4}$, $\varepsilon = 0$, $\Delta x = \frac{1}{256}$, $\Delta t \in \{ 2^{-i};\, 7 \leq i \leq 11\}$).
  b) Rates of convergence for the
  stochastic NLS driven by $W(t) = \sum_{\ell=1}^M \frac{1}{\ell} \sin (\pi \ell x) \beta_\ell(t)$ in the norm $\bigl(E[\Vert \psi(T) - \phi^{N}\Vert^2_{{\mathbb L}^2}]\bigr)^{1/2}$
  ($T = \frac{1}{4}$, $\varepsilon = \sqrt{2}$, $\Delta t \in \{ 2^{-i};\, 7 \leq i \leq 11\}$).}
\end{figure}

\begin{figure}
\centering
  \includegraphics[width=0.5\textwidth]{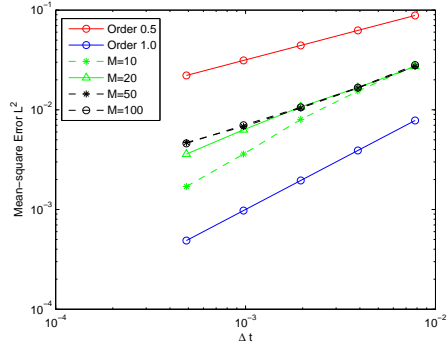}
  \caption{Rates of convergence for the
  stochastic NLS driven by $W(t) = \sum_{\ell=1}^M \frac{1}{\ell^2} \sin (\pi \ell x) \beta_\ell(t)$ in the norm $\bigl(E[\Vert \psi(T) - \phi^{N}\Vert^2_{{\mathbb L}^2}]\bigr)^{1/2}$
  ($T = \frac{1}{4}$, $\varepsilon = \sqrt{2}$, $\Delta t \in \{ 2^{-i};\, 7 \leq i \leq 11\}$).}\label{fig2}
\end{figure}

We consider $\varepsilon=0$ first: Figure 1a) shows order $2$ for the ${\mathbb L}^2$-error of the midpoint scheme. The observations
are different in the stochastic case ($\varepsilon=\sqrt{2}$) where different sorts of Wiener processes depending on $M$
are used: as is displayed in Figure 1b), the strong order of convergence
for $M=1$ and $M=4$ is approximately to $1$, but it drops to $0.5$ approximately for value $8$ of $M$.
As stated in remark 4.6, the convergence order depends on the regularity of the solution which depends on the property of operator ${\bf Q}$, so we consider large $M=50$, the numerical convergence order is only approximate to 0.1.
Consider another spatially more smooth real-valued Wiener process $W(t)=\sum_{\ell=1}^{M}\frac{1}{\ell^2}\sin(\pi\ell x)\beta_{\ell}(t)$, Figure \ref{fig2} shows that the convergence order is approximately to $1$ for $M=10$, and similarly as before it drops to $0.5$ as $M$ grows to $100$.
\begin{figure}[htbp]
  \centering
  \begin{minipage}[b]{5.2 cm}
   \begin{center} a) $\varepsilon = 0$ (charge)  \end{center}
   \centering
    \includegraphics[width=1.078\textwidth]{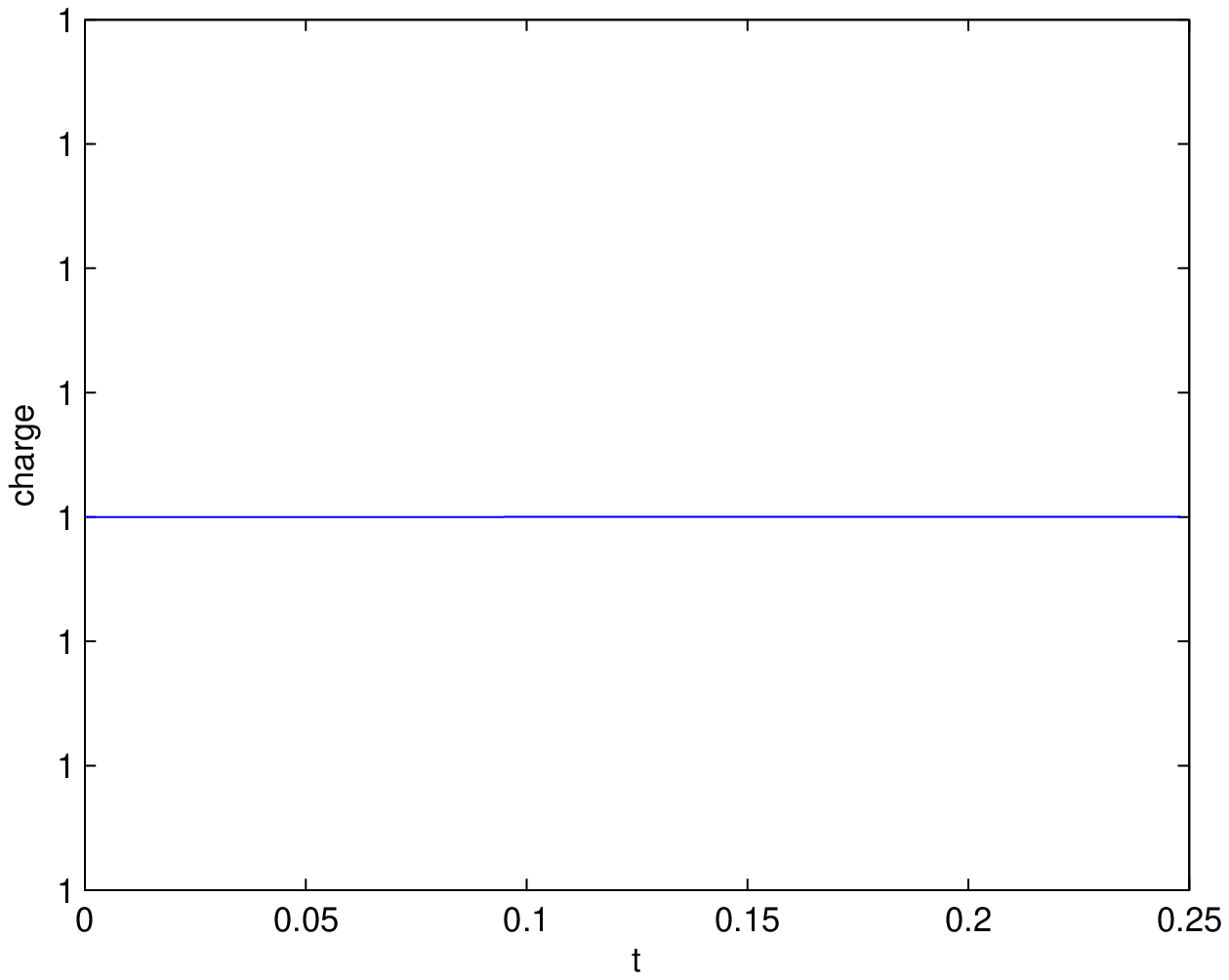}
    \end{minipage}
      \begin{minipage}[b]{5.2 cm}
   \begin{center} b) $\varepsilon = 0$ (energy)  \end{center}
   \centering
    \includegraphics[width=1.078\textwidth]{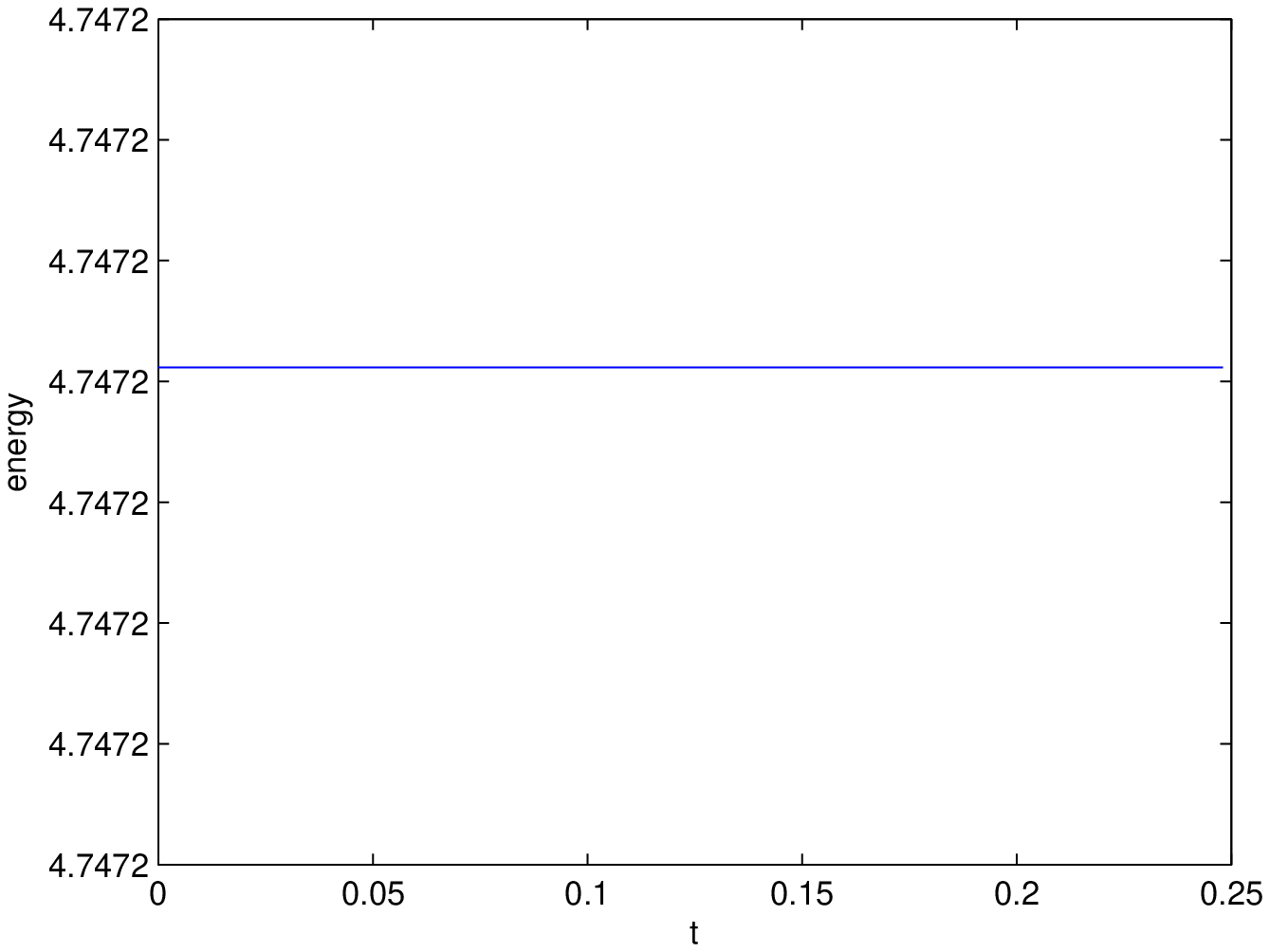}
    \end{minipage}
      \begin{minipage}[b]{5.2 cm}
     \begin{center} c) $\varepsilon = \sqrt{2}$  (charge)  \end{center}
     \centering
    \includegraphics[width=1.078\textwidth]{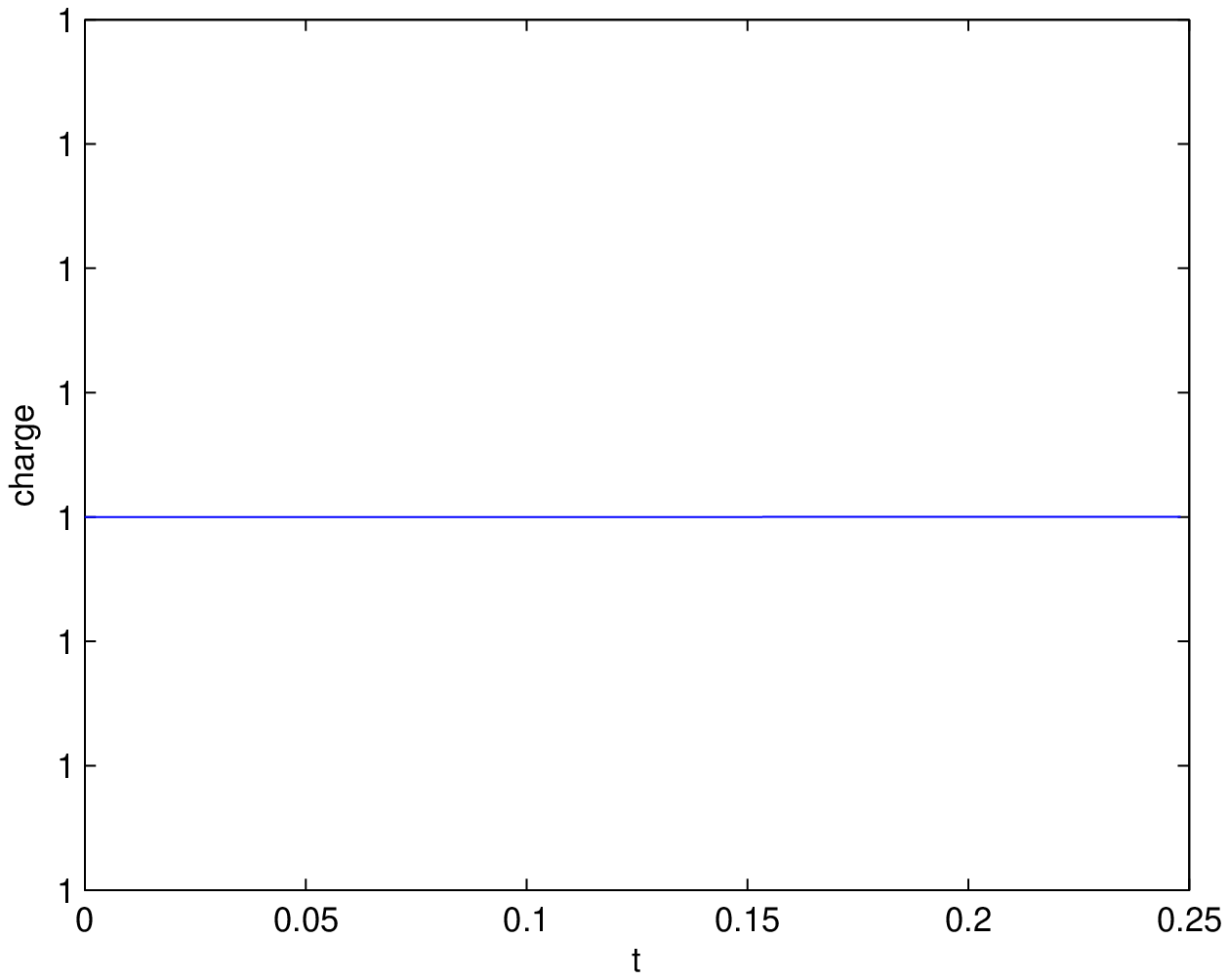}
  \end{minipage}
        \begin{minipage}[b]{5.2 cm}
     \begin{center} d) $\varepsilon = \sqrt{2}$ (energy)   \end{center}
     \centering
    \includegraphics[width=1.078\textwidth]{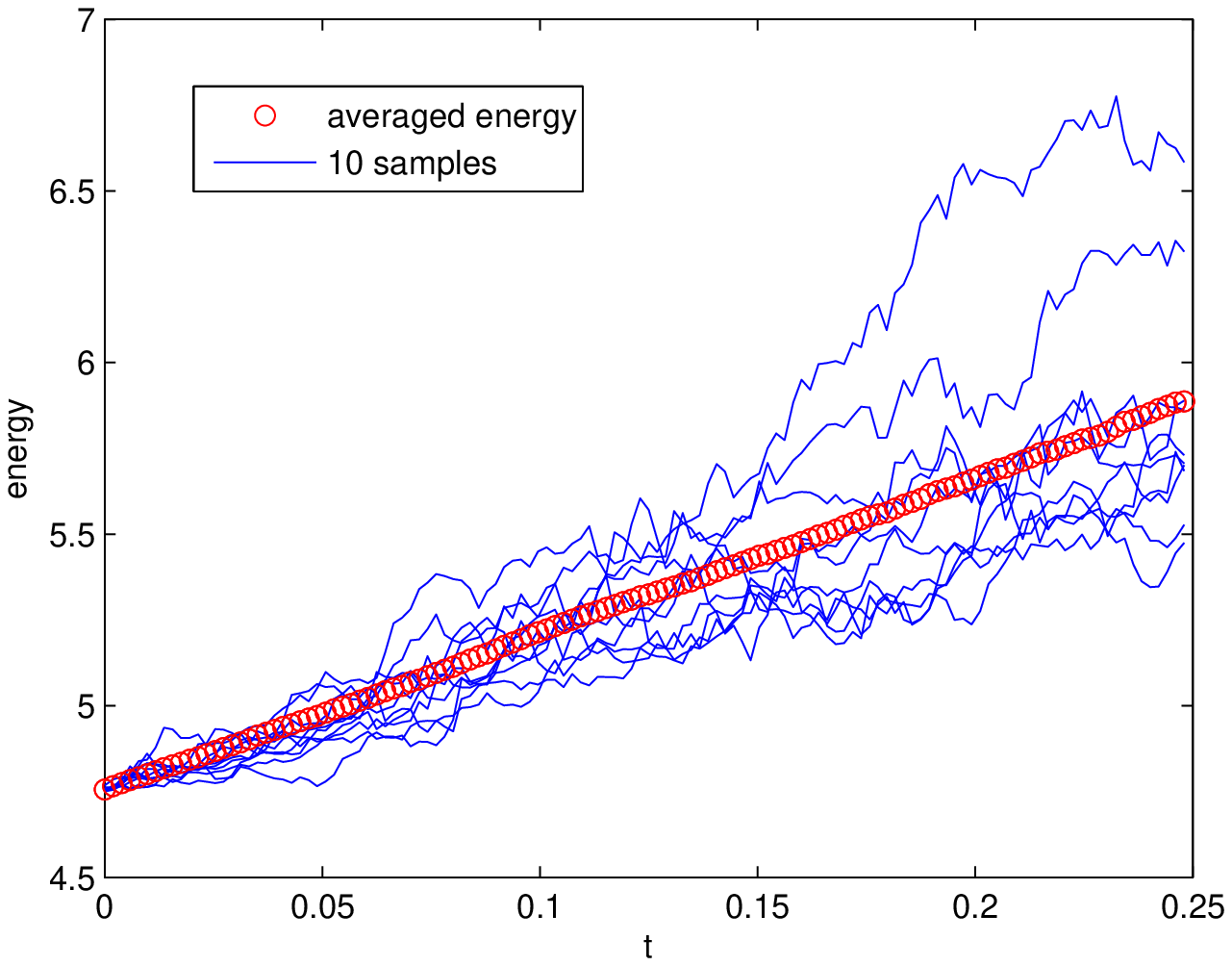}
  \end{minipage}
  \caption{ a) Evolution of charge in deterministic case ($\varepsilon=0$, $\Delta t=2^{-9}$, $T=\frac14$). b) Evolution of energy in deterministic case ($\varepsilon=0$, $\Delta t=2^{-9}$, $T=\frac14$).
  c) Evolution of charge in stochastic case ($\varepsilon=\sqrt{2}$, $\Delta t=2^{-9}$, $T=\frac14$). d) Evolution of energy in stochastic case ($\varepsilon=\sqrt{2}$, $\Delta t=2^{-9}$, $T=\frac14$).}
\end{figure}

The plots in Figures 3a) and 3c) study the conservation of discrete charge of midpoint scheme in both deterministic and stochastic case.
Figures 3b) and 3d) study the evolution of discrete energy of midpoint scheme:  we may observe from Figure 3b) that the discrete energy is preserved by midpoint scheme in the deterministic case; however it is no more a constant in stochastic case, and we observe a linear growth for the discrete averaged energy over $500$ paths.

For comparison with the midpoint scheme \eqref{symp scheme}, we consider a non-structure preserving numerical method
\[
\phi^{n+1}=\phi^{n}+{\bf i}\Delta t A\phi^{n+1} +{\bf i}\Delta t F(\phi^{n+1})-{\bf i}\varepsilon \Delta W_{n},
\]
and apply the same spatial discretization. Figure 4 displays the evolution of discrete charge and energy for the above method: the discrete charge is no more preserved but decrease linearly; the discrete averaged energy is no more linear growth.
\begin{figure}[htbp]
  \centering
  \begin{minipage}[b]{5.2 cm}
   \begin{center} a) $\varepsilon = \sqrt{2}$ (charge)  \end{center}
   \centering
    \includegraphics[width=1.078\textwidth]{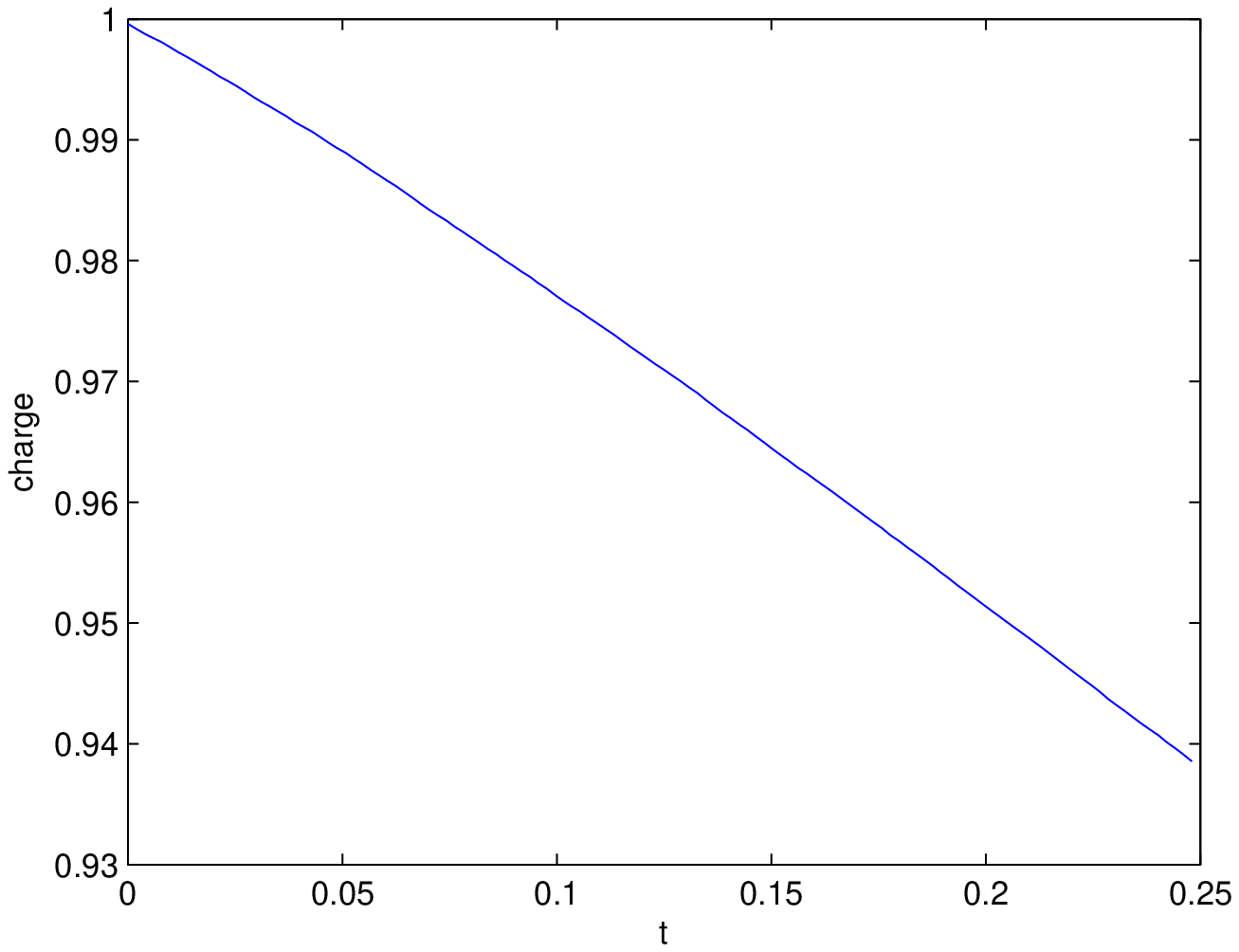}
    \end{minipage}
      \begin{minipage}[b]{5.2 cm}
   \begin{center} b) $\varepsilon = \sqrt{2}$ (energy)  \end{center}
   \centering
    \includegraphics[width=1.078\textwidth]{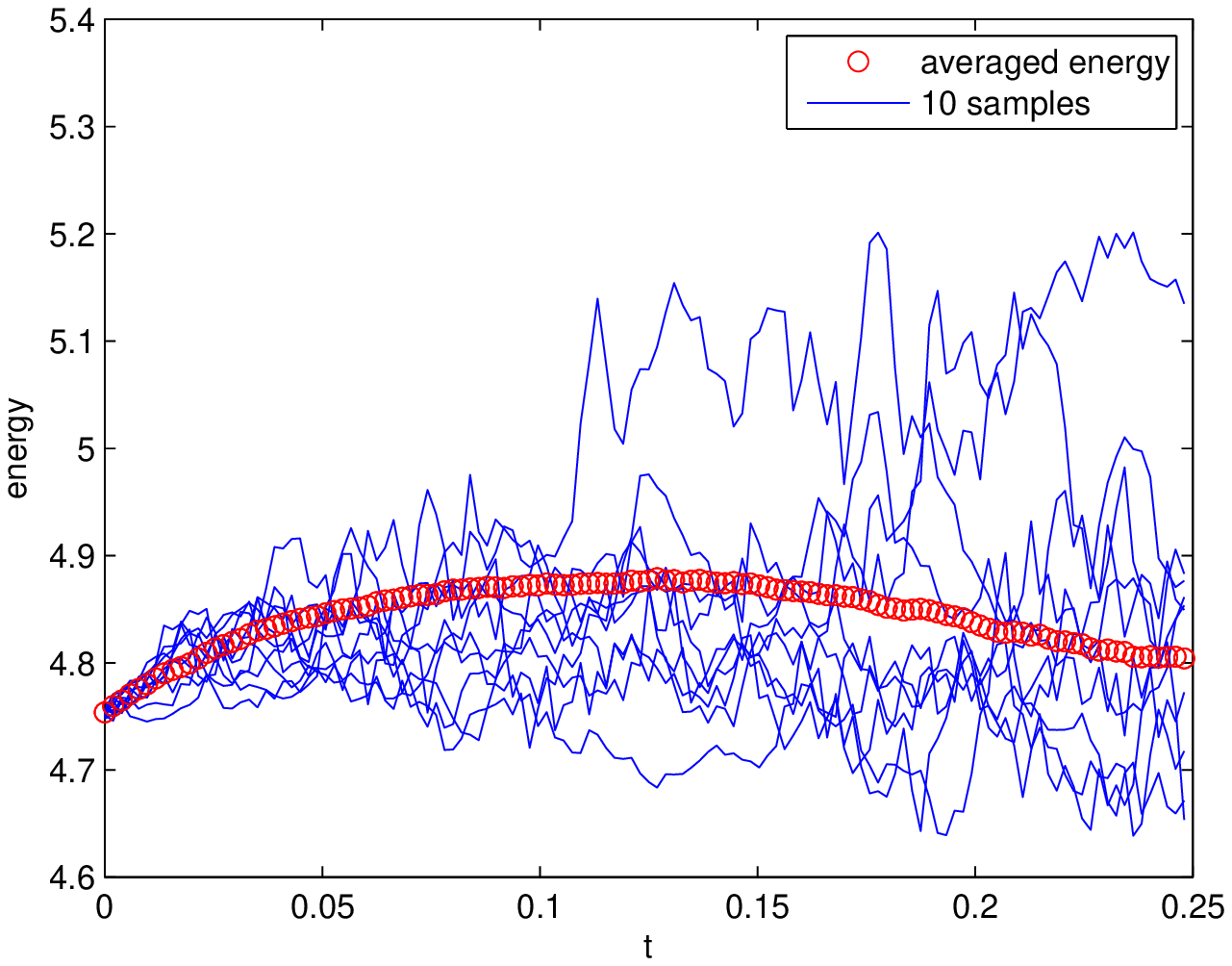}
    \end{minipage}
    \caption{a) Evolution of charge in stochastic case ($\varepsilon=\sqrt{2}$, $\Delta t=2^{-9}$, $T=\frac14$). b) Evolution of energy in stochastic case ($\varepsilon=\sqrt{2}$, $\Delta t=2^{-9}$, $T=\frac14$).}
\end{figure}

\section*{Acknowledgments} Authors would like to thank Prof. Dr. Arnulf Jentzen for helpful suggestions and discussions.

\end{document}